\documentclass[11pt]{article} % use larger type; default would be 10pt
\usepackage[utf8]{inputenc} % set input encoding (not needed with XeLaTeX)

\usepackage{lipsum}
\usepackage{amsfonts}
\usepackage{graphicx}
\usepackage{epstopdf}
\usepackage{algorithmic}
\usepackage{booktabs} % for much better looking tables
\usepackage{array} % for better arrays (eg matrices) in maths
\usepackage{paralist} % very flexible & customisable lists (eg. enumerate/itemize, etc.)
\usepackage{verbatim} % adds environment for commenting out blocks of text & for better verbatim
\usepackage{subfig} % make it possible to include more than one captioned figure/table in a single float
\usepackage{float}
\usepackage{amsmath, amssymb,amsthm} 
\usepackage{color}
\usepackage{hyperref}
\usepackage{mathtools}
\usepackage{mathrsfs}
\usepackage{pgfplots}
\pgfplotsset{compat=1.16}
\ifpdf
  \DeclareGraphicsExtensions{.eps,.pdf,.png,.jpg}
\else
  \DeclareGraphicsExtensions{.eps}
\fi

\newtheorem{thm}{Theorem}[section]
\newtheorem{mydef}[thm]{Definition}
\newtheorem{prop}[thm]{Proposition}

\newtheorem{lem}[thm]{Lemma}
\newtheorem{cor}[thm]{Corollary}

\newtheorem*{nb}{Note}

\title{\bfseries Graph MBO as a semi-discrete implicit Euler scheme for graph Allen\textendash Cahn flow}
\author{Jeremy Budd \qquad Yves van Gennip \vspace*{1em}\\
\small{Delft Institute of Applied Mathematics (DIAM)}\\
\small{Technische Universiteit Delft}\\ \small{Delft, The Netherlands}\vspace*{1em} \\
 \texttt{j.m.budd-1@tudelft.nl} \qquad \texttt{y.vangennip@tudelft.nl}
}
\date{}
\usepackage{amsopn}

\DeclareMathOperator*{\argmin}{argmin}
\DeclareMathOperator*{\Glim}{\Gamma-lim}

\DeclareMathOperator{\TV}{TV}
\DeclareMathOperator{\GL}{GL_{\varepsilon}}

\DeclareMathOperator{\sgn}{sgn}

\DeclarePairedDelimiter\ceil{\lceil}{\rceil}

\DeclarePairedDelimiter\ip{\langle}{\rangle_\V}

\usepackage{geometry} % to change the page dimensions
\usepackage{graphicx} % support the \includegraphics command and options

\numberwithin{equation}{section}
\newcommand{\be}{\begin{equation}}
\newcommand{\ee}{\end{equation}}

\newcommand{\V}{\mathcal{V}}
\newcommand{\bigO}{\mathcal{O}}
\begin{document}
\maketitle
\begin{abstract}
In recent years there has been an emerging interest in PDE-like flows defined on finite graphs, with applications in clustering and image segmentation. In particular for image segmentation and semi-supervised learning Bertozzi and Flenner [A. L. Bertozzi and A. Flenner, \emph{Multiscale Modeling and Simulation}
10 (2012), no. 3, pp. 1090\textendash1118] developed an algorithm based on the Allen--Cahn gradient flow of a graph Ginzburg--Landau functional, and Merkurjev, Kosti\'c and Bertozzi [E. Merkurjev, T. Kostić, and A. L. Bertozzi, \emph{SIAM Journal on Imaging Sciences} 6 (2013), no. 4, pp. 1903\textendash 1930] devised a variant algorithm based instead on graph Merriman--Bence--Osher (MBO) dynamics. 
This work offers rigorous justification for this use of the MBO scheme in place of Allen--Cahn flow. First, we choose the double-obstacle potential for the Ginzburg--Landau functional, and derive well-posedness and regularity results for the resulting graph Allen--Cahn flow. Next, we exhibit a ``semi-discrete" time-discretisation scheme for Allen--Cahn flow of which the MBO scheme is a special case. We investigate the long-time behaviour of this scheme, and prove its convergence to the Allen--Cahn trajectory as the time-step vanishes. Finally, following a question raised by Van Gennip, Guillen, Osting and Bertozzi [Y. van Gennip, N. Guillen, B. Osting, and A. L. Bertozzi,
\emph{Milan Journal of Mathematics} 82 (2014), no. 1, pp. 3\textendash65], we exhibit results towards proving a link between double-obstacle Allen--Cahn flow and mean curvature flow on graphs. We show some promising $\Gamma$-convergence results, and translate to the graph setting two comparison principles used by Chen and Elliott [X. Chen and C. M. Elliott, \emph{Proc. Mathematical and Physical Sciences} 444 (1994), no. 1922, pp. 429\textendash 445] to prove the analogous link in the continuum.


\textbf{Classifications:}
05C99, 34B45, 35R02, 53C44, 49K15, 35K05, 34A12, 65N12 
\end{abstract}
\section{Introduction}
In this paper, we derive a link between formulations of the Merriman--Bence--Osher (MBO) scheme for diffusion generated motion and the Allen--Cahn gradient flow of the Ginzburg--Landau functional on finite graphs. We go on to discuss results towards a link between these flows and a graph formulation of mean curvature flow.

The core background for this work is the paper of Van Gennip, Guillen, Osting and Bertozzi \cite{vGGOB} in which a framework for graph-based analysis is defined, and within this framework described graph variants of MBO, Allen--Cahn and mean curvature flow. This paper builds upon that framework, seeking to demonstrate concrete links between these flows, especially in light of their interrelated use in algorithms developed by Bertozzi and co-authors \cite{BF,MKB} inspired by the connections between mean curvature flow and the method of Chan--Vese (see e.g. \cite{ET}). %[[DESCRIBE MORE HOW WE RELATE AND DO NOT RELATE TO 1]]

The central result of this paper is that taking the Ginzburg--Landau functional with the ``double-obstacle" potential  (see Blowey and Elliott \cite{BE1991,BE1992,BE1993} for detail in the continuum context, and Bosch, Klamt and Stoll \cite{BKS2018} in the graph context) the MBO scheme becomes exactly a particular choice of time-step for a ``semi-discrete" numerical scheme for the Allen--Cahn flow. 
%%We will explore the properties of our semi-discrete scheme (and thus in particular MBO) and its convergence to the continuous-time Allen--Cahn flow with this potential. Furthermore, we prove existence, uniqueness and Lipschitz regularity for this Allen--Cahn flow. 
%Finally, we follow \cite{vGGOB} in investigating links between MBO, Allen--Cahn and a graph formulation of mean curvature flow. We present encouraging $\Gamma$-convergence results, and prove a pair of comparison principles that are graph analogues of comparison principles used by Chen and Elliott \cite{CE} to prove convergence of continuum Allen--Cahn (with double-obstacle potential) to mean curvature flow.  
\subsection{Contributions of this work}
In this paper we have:
\begin{itemize}
\item Defined a graph Allen--Cahn flow with double-obstacle potential (Definition \ref{ACdef}) along with an explicit form and a weak form (Theorems \ref{ACexplicit} and \ref{ACobsweak} respectively).
\item Proved well-posedness (Theorems \ref{existence}, \ref{uniquenessprequel}, and \ref{wellp}), monotonicity of Ginzburg--Landau (Proposition \ref{ACgf}), and Lipschitz regularity (Theorem \ref{ACLips}) of solutions to this flow.
\item Defined the semi-discrete scheme for this flow (Definition \ref{SDdef}) and proved that the MBO scheme is a special case of the semi-discrete scheme (Theorem \ref{obsMMthm}).
\item Derived a Lyapunov functional for the semi-discrete scheme (Theorem \ref{Lyapthm}) and proved convergence of the scheme to the Allen--Cahn trajectory (Theorem \ref{SDlimit}).
\item Proved the $\Gamma$-convergence of this Lyapunov functional to total variation (Corollary \ref{LyapGamma}), and graph variants of two key comparison principles from \cite{CE} (Theorems \ref{cp2} and \ref{cp1}).
\end{itemize}
Beyond the immediate contributions of this paper, we note two major avenues this work opens for future research:
\begin{itemize}
\item The framework of this paper yields a natural way of incorporating constraints on the MBO scheme by imposing them on the Allen--Cahn flow, and then investigating the semi-discrete scheme for this modified Allen--Cahn flow. Ongoing work by the authors will apply this strategy to investigating Allen--Cahn flow and the MBO scheme under mass-conservation \cite{BuddvG} and fidelity-forcing \cite{Buddfidelity} constraints.
\item The semi-discrete scheme yields a family of MBO-like schemes that may prove fruitful as alternatives to MBO and Ginzburg--Landau methods in applications. Ongoing work by the authors will investigate this for image processing and classification in \cite{Buddfidelity}.
\end{itemize}
\subsection{Background}
In the continuum, it is well-known that Allen--Cahn, MBO and mean curvature flow share important interrelations. The MBO scheme was developed in \cite{MBO1992} as a means of approximating motion according to mean curvature flow by iterative diffusion and thresholding of a set. The paper gave a formal analysis showing that diffusion of a set locally corresponded to motion with curvature dependent velocity, suggesting a convergence as the MBO time-step went to zero. This formal analysis was then supported by rigorous convergence proofs by Evans \cite{Evans} and Barles and Georgelin \cite{BG}. Recently, Swartz and Kwan Yip \cite{SKY} presented an elementary proof of the convergence via the weak formulation of mean curvature flow in \cite{LS}. The connections between Ginzburg\textendash Landau dynamics and mean curvature flow have been extensively studied, dating back to a formal analysis by Allen and Cahn in \cite{AC}. The basic convergence result, see e.g. \cite{BV}, \cite{ESS} and \cite{Soner}, is that as $\varepsilon\rightarrow 0$ the Allen\textendash Cahn solution tends to a phase-separation with the interface evolving by mean curvature flow. Thus a method of approximating mean curvature flow is as a singular limit of ``phase fields" evolving under the Allen\textendash Cahn equation.\footnote{See \cite{BV}, \cite{ESS} and \cite{BSS} for detail on this method.}    

Mean curvature flow also arises in a discrete context, e.g. in image segmentation. A major technique in this area is a variational approach, involving minimising the Mumford\textendash Shah functional \cite{MS}. As this functional is quite intensive to minimise in full generality, Chan and Vese \cite{CV} introduced a method of using a level-set approach in the simplified case of a piecewise-constant image. In \cite{ET} Esedo\=glu and Tsai considered in particular the case where the image $u$ takes just two values $c_1$ and $c_2$ on regions $\Sigma$ and $\Omega\setminus\Sigma$ respectively.
%, leading to the functional: \be\label{CV} \mathscr{E}(\Sigma,c_1,c_2) := \lambda \int_\Sigma (c_1-f)^2 \; dx +\lambda \int_{\Omega\setminus\Sigma} (c_2-f)^2 \; dx + |\partial\Sigma|.  \ee 
The Euler\textendash Lagrange level-set equations devised to minimise the Chan--Vese functional in this case closely resemble those associated to mean curvature flow of the boundary $\partial\Sigma$.\footnote{In particular, a common approach leads to motion of $\partial\Sigma$ with normal velocity $\kappa - \lambda(c_1-f)^2 +\lambda(c_2-f)^2$ where $\kappa$ is the mean curvature, $f$ the reference and $\lambda$ the strength of fidelity to the reference. } Motivated by this and the convergence results above,  Esedo\=glu and Tsai devised a variant of the MBO scheme to minimise the functional. Interpreting the MBO scheme as an approximation to a time-splitting of the Allen\textendash Cahn equation, an idea we will return to in our analysis in this paper, they consider a modified Mumford--Shah energy using the Ginzburg--Landau function in place of total variation and devise an MBO-like scheme based on a modified diffusion followed by a thresholding to minimise this energy.

Inspired by these techniques, in \cite{BF} Bertozzi and Flenner embraced the discrete nature of an image and devised a discrete graph-based method for image segmentation (and related topics, such as semi-supervised learning) using the graph Ginzburg\textendash Landau functional (with symmetric normalised Laplacian). In \cite{MKB} Merkurjev, Kosti\'c and Bertozzi developed a faster variant of this method by employing the MBO scheme on a graph, motivated by the link between continuum Ginzburg\textendash Landau and MBO through their association with mean curvature flow. They also extended this method to apply to non-local image inpainting. An example of an application of these techniques is recent work by Calatroni, Van Gennip, Sch{\"o}nlieb, Rowland and Flenner \cite{birdspot}.

The use of these methods implicitly assumes that the continuum connections between these processes extend to their graph counterparts. But a challenge to this assumption is that graphs need not resemble the continuum objects to which the above convergence results apply. For example there has been interest, though not to the authors' knowledge any published work, in applying the above graph methods to social networks\textemdash which need not resemble meshes on continuum manifolds in structure. We seek to investigate rigorously the validity of this assumption. 

Our framework for this investigation is that developed in \cite{vGGOB}. In that paper, Van Gennip, Guillen, Osting and Bertozzi formally defined the key concepts for analysis on graphs, and defined the Allen--Cahn flow, MBO scheme, and mean curvature flow on a graph. They proved rigorous results about each of these flows individually, particularly about the conditions under which these flows ``pin'' (or ``freeze"). 
Finally, they briefly raise (but do not prove any results towards) a number of questions for future research, including whether and how the three flows are linked. 
This paper will present substantial new work towards answering some of these questions. 
%[[DESCRIBE 1 IN MORE DETAIL]]
%TEST TEST TEST TEST TEST TEST TEST TEST TEST TEST TEST TEST TEST TEST
%TEST TEST TEST TEST TEST TEST TEST TEST TEST TEST TEST TEST TEST TEST
%TEST TEST TEST TEST TEST TEST TEST TEST TEST TEST TEST TEST TEST TEST
%TEST TEST TEST TEST TEST TEST TEST TEST TEST TEST TEST TEST TEST TEST
%TEST TEST TEST TEST TEST TEST TEST TEST TEST TEST TEST TEST TEST TEST
\subsection{Paper outline} We here give a brief overview of the rest of this paper.

In section \ref{groundwork}, we introduce our notation for analysis on graphs, introduce the graph Allen--Cahn flow and MBO scheme, and sketch non-rigorously the link we will derive in this paper.

In section \ref{ACsec}, we define rigourously graph Allen--Cahn flow with the double obstacle potential, and prove various results concerning this flow, relegating some proofs to appendices \ref{existsec} and \ref{uniquesec}. 

In section \ref{SDsec}, we define the semi-discrete scheme for this flow, and prove that the MBO scheme is a special case of it. We derive a Lyapunov functional for this scheme, investigate the long-time behaviour, and relate the scheme to a time-splitting of the Allen--Cahn flow.

In section \ref{SDconvsec}, we prove that trajectories of the semi-discrete scheme converge pointwise to Allen--Cahn trajectories as the time-step converges to zero.

Finally, in section \ref{MCFsec} (and appendices \ref{uniquesec} and \ref{compsec}) we discuss linking graph mean curvature flow to the graph MBO scheme and Allen--Cahn flow. We present encouraging $\Gamma$-convergence results, and prove graph analogues of a pair of comparison principles from Chen and Elliott's \cite{CE} convergence proof of continuum double-obstacle Allen--Cahn flow to mean curvature flow.  
\section{Groundwork}\label{groundwork}
The framework for analysis on graphs is presented in \cite{vGGOB}, we reproduce here those aspects needed for our discussion. Let $G=(V,E)$ be a finite, undirected, weighted, simple and connected graph with vertex set $V$, edge set $E\subseteq V^2$, and positive weights $\{\omega_{ij}\}_{ij\in E}$ with $\omega_{ij}=\omega_{ji}$. We extend $\omega_{ij}=0$ when $ij\notin E$. On $G$ we define the spaces ($X\subseteq \mathbb{R}$): 
\begin{align*}
	&\V := \left\{ u: V\rightarrow\mathbb{R} \right\} , &\V_{X} := \left\{ u: V\rightarrow X \right\}&,  &\mathcal{E} := \left\{ \varphi: E\rightarrow\mathbb{R} |\varphi_{ij} = -\varphi_{ji} \right\}.&
\end{align*}
Since $V$ is finite, we shall interchangeably view these as functions and as real vectors. Next, we define the spaces of time-dependent vertex functions (where $T\subseteq\mathbb{R}$ an interval)
\begin{align*}
	&\V_{t\in T} := \left\{ u: T\rightarrow\V \right\} , &\V_{X,t\in T} := \left\{ u: T\rightarrow \V_X \right\}.&
\end{align*} For a parameter $r\in [0,1]$, and denoting $d_i:=\sum_j \omega_{ij}$, which we refer to as the \emph{degree} of vertex $i$,  we define the following inner products on $\V$ and $\mathcal{E}$:\footnote{In this paper we take, so as to preserve the $\Gamma$-convergence result in \cite{vGB}, $q=1$ in the definitions in \cite{vGGOB}.}
\begin{align*}
	&\ip{u,v} := \sum_{i\in V} u_i v_i d_i^r, &\langle\varphi,\phi\rangle_\mathcal{E}:=\frac{1}{2}\sum_{i,j\in V} \varphi_{ij} \phi_{ij}\omega_{ij}&\end{align*} and define the inner product on $\V_{t\in T}$ (or $\V_{X,t\in T}$): \[ (u,v)_{t\in T}:=\int_T \left\langle u(t),v(t)\right\rangle_\V \;dt = \sum_{i\in V} d_i^r (u_i,v_i)_{L^2(T;\mathbb{R})}\] where $(\cdot,\cdot)_{L^2(T;\mathbb{R})}$ is the standard continuum $L^2$ inner product.
These induce norms $||\cdot||_\V$, $||\cdot||_\mathcal{E}$, and $||\cdot||_{t\in T}$ in the usual way. We also define for $u\in \V$ the norm $ ||u||_\infty := \max_{i\in V} |u_i|$.
We next define the $L^2$ space: \[L^2(T;\V) : =\left\{ u\in\V_{t\in T} \, | \, ||u||_{t\in T} <\infty \right\}.\]
Finally, for $T$ an open interval, we define the \emph{Sobolev space} $H^1(T;\V)$ as the set of $u \in L^2(T;\V)$ with generalised time derivative $du/dt \in L^2(T;\V)$ such that  
\[  \forall \varphi\in C^\infty_c(T;\V)\:\:\left(u,\frac{d\varphi}{dt}\right)_{t\in T} = -\left(\frac{du}{dt},\varphi\right)_{t\in T} \] 
where $C^\infty_c(T;\V)$ denotes the infinitely differentiable and compactly supported elements of $\V_{t\in T}$. 
We link this to the familiar continuum $H^1$.
\begin{prop}\label{graphH1}
$u\in H^1(T;\V)$ if and only if $u_i \in H^1(T;\mathbb{R})$ for each $i\in V$.
\end{prop}
\begin{proof}Note that $(du/dt)_i = du_i/dt$, so $u$ and $du/dt\in L^2(T;\V)$ if and only if $\forall i\in V$, $u_i$ and $du_i/dt\in L^2(T;\mathbb{R})$. 
Next, $(u,d\varphi/dt)_{t\in T} =  -(du/dt,\varphi)_{t\in T}$ if and only if \[\sum_{i\in V} d_i^r (u_i,d\varphi_i/dt)_{L^2(T;\mathbb{R})}=-\sum_{i\in V} d_i^r (du_i/dt,\varphi_i)_{L^2(T;\mathbb{R})}.\]
It follows that $\forall \varphi\in C^\infty_c(T;\V) \; (u,d\varphi/dt)_{t\in T} =  -(du/dt,\varphi)_{t\in T}$ if and only if \[ \forall i\in V\;\forall \phi\in C^\infty_c(T;\mathbb{R}) \:\:(u_i,d\phi/dt)_{L^2(T;\mathbb{R})}=-(du_i/dt,\phi)_{L^2(T;\mathbb{R})}\] and therefore $\forall i\in V$ $u_i\in H^1(T;\mathbb{R})$.
\end{proof}
We define the following inner product on $H^1(T;\V)$:\[ (u,v)_{H^1(T;\V)} := (u,v)_{t\in T} + \left(\frac{du}{dt},\frac{dv}{dt}\right)_{t\in T} = \sum_{i\in V} d_i^r (u_i,v_i)_{H^1(T;\mathbb{R})}.\] We also define the local $H^1$ space \[ H^1_{loc}(T;\V) :=\left\{u\in \V_{t\in T}\,\middle|\,\forall a,b\in T, \: u\in H^1((a,b);\V) \right\}\] and we likewise define $L^2_{loc}(T;\V)$. 

Next, we introduce the graph variants of familiar vector calculus gradient and Laplacian:
\begin{align*}
	&(\nabla u)_{ij}:=\begin{cases}u_j -u_i, & ij\in E\\ 0, &\text{otherwise} \end{cases} &(\Delta u)_i:=d_i^{-r}\sum_{j\in V}\omega_{ij}(u_i-u_j)&
\end{align*}
where the graph Laplacian\footnote{Our choice of $r$ dictates which graph Laplacian we use. For $r=0$ we have $\Delta = D -A$ the standard \emph{unnormalised Laplacian}. For $r=1$ we have $\Delta=I-D^{-1}A$ the \emph{random walk Laplacian}. Note that the \emph{symmetric normalised Laplacian} $I-D^{-1/2}AD^{-1/2}$ used in \cite{BF,MKB} is not covered by our scheme.} $\Delta$ is positive semi-definite, unlike the negative semi-definite continuum Laplacian. From this we define the \emph{graph diffusion operator} \[e^{-t\Delta}u:=\sum_{n\geq 0} \frac{(-1)^n t^n}{n!}\Delta^n u\] where $v(t)=e^{-t\Delta}u $  is the unique solution to the diffusion equation \[ \frac{dv}{dt} = -\Delta v, \: v(0) = u.\] We recall the familiar functional analysis notation, for some $F:\V\rightarrow\V$, of
\begin{align*} &\rho(F):=\max\{|\lambda| :\text{$\lambda$ an eigenvalue of $F$}\}\\
&||F|| := \sup_{||u||_\V = 1} ||Fu||_\V
\end{align*}
and recall the standard result that if $F$ is self-adjoint then $||F|| = \rho(F)$.\begin{prop}
	If $u\in H^1(T;\V)$ and $T$ bounded below, then $e^{-t\Delta}u\in H^1(T;\V)$ with \[\frac{d}{dt}\left(e^{-t\Delta}u\right) = e^{-t\Delta}\frac{du}{dt} -  e^{-t\Delta}\Delta u.\]
\end{prop}\begin{proof} Let $T =(a,b)$ with $a>-\infty$.  Now, $e^{-t\Delta}$ has eigenvalues $e^{-\lambda_k t}$, for $\lambda_k\geq 0$ the eigenvalues of $\Delta$, and $e^{-t\Delta}$ is self-adjoint so $||e^{-t\Delta}||=\rho(e^{-t\Delta})\leq \max\left\{1,e^{-a||\Delta||}\right\}$ for $t\in T$. So $e^{-t\Delta}$ is a uniformly bounded operator for $t\in T$ and therefore $||e^{-t\Delta}u||_{t\in T}<\infty$ and $||e^{-t\Delta}\frac{du}{dt} -  e^{-t\Delta}\Delta u||_{t\in T}<\infty$. Then for $\varphi\in C^\infty_c(T;\V)$ \begin{align*}
	\left(e^{-t\Delta}\frac{du}{dt} -  e^{-t\Delta}\Delta u,\varphi\right)_{t\in T} &= \left(\frac{du}{dt},e^{-t\Delta}\varphi\right)_{t\in T} - \left(u,e^{-t\Delta}\Delta\varphi\right)_{t\in T}\\
	&= -\left(u,\frac{d}{dt}\left(e^{-t\Delta}\varphi\right)+e^{-t\Delta}\Delta\varphi\right)_{t\in T}\\
	&= - \left(u,e^{-t\Delta}\frac{d\varphi}{dt}\right)_{t\in T} = - \left(e^{-t\Delta}u,\frac{d\varphi}{dt}\right)_{t\in T}
\end{align*}
so $e^{-t\Delta}u$ has the desired generalised derivative.
\end{proof}

Finally, when considering variational problems of the form \[ \underset{x}{\argmin}\: f(x) \] we will write $f \simeq g$ and say the functionals are \emph{equivalent} when $g(x) = af(x) + b$ for $a,b$ independent of $x$ and $a>0$. This ensures that $f$ and $g$ have the same minimisers.

We now define the two processes we wish to link in this paper.
\begin{mydef}[Graph MBO scheme] The \emph{graph Merriman\textendash Bence\textendash Osher scheme (MBO)} is a scheme that creates a series of vertex sets \emph{(}i.e. elements of $\V_{\{0,1\}}$\emph{)} by iterating two steps. First, diffuse the characteristic function $\chi_{S_n}$ of the set $S_n\subseteq V$ for a time $\tau$ to yield $v = e^{-\tau\Delta}\chi_{S_n}$. Second, threshold to define $S_{n+1} = \{i\in V|v_i\ge 1/2\}$. In \emph{\cite[Proposition 4.6]{vGGOB}} it was shown that this scheme can be expressed variationally, where $u_n=\chi_{S_n}$, by 
\begin{equation}
	\label{MBO}
	u_{n+1}\in \underset{u\in\V_{[0,1]}}{\argmin}\: \left\langle \mathbf{1}-2e^{-\tau\Delta}u_n,u\right\rangle_\V
\end{equation} 
\emph{(}where $\mathbf{1}$ is the vector of ones\emph{)} which we can rewrite with the equivalent functional\emph{:}\footnote{One can check that $\left\langle \mathbf{1}-2e^{-\tau\Delta}u_n,u\right\rangle_\V=\ip{u,\mathbf{1}-u}+\ip{u-e^{-\tau\Delta}u_n,u-e^{-\tau\Delta}u_n}-\ip{e^{-\tau\Delta}u_n,e^{-\tau\Delta}u_n}$. Then suppress the constant (in $u$) term $\ip{e^{-\tau\Delta}u_n,e^{-\tau\Delta}u_n}$ and divide by $2\tau$.}
\begin{equation}
		\label{MBOMM}
	u_{n+1}\in\underset{u\in\V_{[0,1]}}{\argmin}\: \frac{1}{2\tau}\ip{\mathbf{1}-u,u} + \frac{\left|\left|u-e^{-\tau\Delta}u_n\right|\right|^2_\V}{2\tau}. 
\end{equation}
\end{mydef}
\begin{nb}
Note that \eqref{MBOMM} has a form resembling a \emph{discrete solution} \emph{\cite[Definition 2.0.2]{AGS2008}  (}cf. the study of minimising movements\emph{)} of a gradient flow. That is, it resembles a sequence arising from an Euler scheme for a gradient flow. This motivates our link to Allen\textendash Cahn flow.
\end{nb}
\begin{mydef}[Graph Allen--Cahn flow]
\emph{Graph Allen\textendash Cahn (AC) flow} is the $\ip{\cdot,\cdot}$ gradient flow of the \emph{graph Ginzburg\textendash Landau functional}, which we shall define as\emph{:} 
\begin{equation}
	\label{GL}
	\GL(u) := \frac{1}{2}\left|\left|\nabla u\right|\right|_\mathcal{E}^2 +\frac{1}{\varepsilon}\left\langle W\circ u,\mathbf{1}\right \rangle_\V
\end{equation}
where $W$ is a double-well potential. This definition slightly differs from that in \emph{\cite{vGGOB}:} we have replaced their $\sum_i W(u_i)$ with $\left\langle W\circ u,\mathbf{1}\right \rangle_\V$, which we have found plays better with the Hilbert space structure and enables the link we derive with the MBO scheme. The AC flow is then given, for $W$ differentiable, by the ODE\emph{:}
\begin{equation}
	\label{AC}
	\frac{du}{dt} = -\Delta u - \frac{1}{\varepsilon} W'\circ u = -\nabla_\V\GL(u)
\end{equation}
where $\nabla_\V$ is the Hilbert space gradient on $\V$. 
\end{mydef}
We now describe informally the link which we shall make rigorous in the rest of this paper. To link the AC flow to the MBO scheme, we must discretise it in time. Note that the MBO scheme, although discrete in time, thresholds after a continuous-time diffusion. To capture this, we introduce what we term a \emph{semi-discrete} implicit Euler scheme:
\begin{equation}
	\label{sdAC}
	u_{n+1}=e^{-\tau\Delta} u_n- \frac{\tau}{\varepsilon} W'\circ u_{n+1}.
\end{equation}
To link to \eqref{MBOMM} we express \eqref{sdAC} as a discrete solution. We rewrite \eqref{sdAC} using the Hilbert space gradient as: 
\[ 0 = \nabla_\V|_{u=u_{n+1}}\left( \frac{1}{\varepsilon}\left\langle W\circ u,\mathbf{1}\right \rangle_\V + \frac{\left|\left|u-e^{-\tau\Delta}u_n\right|\right|^2_\V}{2\tau}
\right)\] suggesting that solutions to our semi-discrete scheme obey the variational equation: \begin{equation}
	\label{ACMM}
	\begin{split}
	u_{n+1}\in \underset{u\in \V}{ \argmin }  \: &\frac{1}{\varepsilon}\left\langle W\circ u,\mathbf{1}\right \rangle_\V + \frac{\left|\left|u-e^{-\tau\Delta}u_n\right|\right|^2_\V}{2\tau}.
	%\\							   \simeq
							   %\footnotemark\:&\frac{1}{\varepsilon}\left\langle W\circ u,1\right \rangle_\V +\left\langle u,\frac{u_n-e^{-\tau\Delta}u_n}{\tau}\right \rangle_\V + \frac{\left|\left|u-u_n\right|\right|^2_\V}{2\tau}.
	\end{split}
\end{equation}%\footnotetext{By the Pythagorean identity $\left|\left|u-A\right|\right|^2_\V= \left|\left|u-B\right|\right|^2_\V + 2\ip{u,B-A} -2\ip{B,B-A} + \left|\left|B-A\right|\right|^2_\V \\ \simeq \left|\left|u-B\right|\right|^2_\V + 2\ip{u,B-A}$, suppressing terms constant in $u$, and so the equivalence follows.}
Now, note for $\varepsilon=\tau$ the striking similarity between the functionals in \eqref{MBOMM} and \eqref{ACMM}, so long as we choose a suitable $W$. Inspecting \eqref{MBOMM} suggests taking $W(x) = \frac{1}{2}x(1-x)$. But this does not have two wells (or indeed, any minima). We also wish to force minimisers of \eqref{ACMM} to lie in $\V_{[0,1]}$ as in \eqref{MBOMM}. A potential that satisfies these demands is the well-known \emph{double-obstacle potential}
%\footnote{The use of the double-obstacle potential in graph Allen\textendash Cahn (with fidelity) has also recently been explored by Bosch, Klamt and Stoll \cite{BKS2018}.}
studied extensively by Blowey and Elliott \cite{BE1991,BE1992,BE1993}: \begin{equation}
	\label{Wobs}
	W(x) := \begin{cases}
    \frac{1}{2}x(1-x), & \text{for } 0 \leq x \leq 1, \\
    \infty, & \text{otherwise.}  \end{cases}
\end{equation}
%\hrule
In the remainder of this paper $W$ will denote this double-obstacle potential. This discussion then suggests that for $\varepsilon=\tau$ the semi-discrete scheme for the AC flow \emph{is} the MBO scheme.
%This paper will prove and explore the following theorem.
%\begin{thm}
%\label{thm1}
%Taking $\varepsilon=\tau$ and choosing $W$ as in \eqref{Wobs}, we get that solutions to our semi-discrete scheme \eqref{SDobs}, which is \eqref{sdAC} refined to apply to a non-differentiable $W$, obey the variational equation\emph{:}
%\be
% \begin{split}
%	u_{n+1}\in \underset{u\in \V_{[0,1]}}{ \argmin }  \: &\left\langle u,\mathbf{1}-u\right \rangle_\V +  \left|\left|u-e^{-\tau\Delta}u_n\right|\right|^2_\V\\
%							   %=\:&\ip{u,1}-\ip{u,u}+\ip{u,u}-2\left\langle u,e^{-\tau\Delta}u_n\right \rangle_\V + \left\langle e^{-\tau\Delta}u_n,e^{-\tau\Delta}u_n\right \rangle_\V\\
%							\simeq\:&\left\langle u,\mathbf{1}-2e^{-\tau\Delta}u_n\right \rangle_\V
%\end{split}
%\ee
%and thus the solutions correspond exactly to MBO trajectories.
%\end{thm}
\section{Allen\textendash Cahn evolution with a double-obstacle potential}\label{ACsec}
However, our definition of the AC flow in \eqref{AC} assumed that $W$ was differentiable, which of course the double-obstacle potential is not at 0 and 1. 
Towards extending our definition, write \[W(x) = \frac{1}{2}x(1-x) + I_{[0,1]}(x)\] where $I_{[0,1]}$ is the indicator function taking value 0 on $[0,1]$ and $\infty$ elsewhere. Now, following \cite{BE1993} we seek $H^1_{loc}$ solutions $u$ to \eqref{AC} rewritten using the subdifferential as:
\begin{equation}
	\label{ACobs}
	-\frac{du}{dt}-\Delta u \in \frac{1}{\varepsilon} \partial W(u).
\end{equation}
That is, for almost every $t$ in some chosen interval $T$, and every $i\in V$, we desire
\begin{equation}
\label{ACobs1}
	\varepsilon \frac{du_i}{dt} + \varepsilon(\Delta u(t))_i  +\frac{1}{2}-u_i(t)= \beta_i(t)\in -\partial I_{[0,1]}(u_i(t))
\end{equation}
where the condition on $\beta$ can be written more transparently as
\[ \beta_i(t)\in\begin{cases}
		\{\infty\}, & u_i<0,\\
		[0,\infty), & u_i(t)=0,\\
		\{0\}, &0<u_i(t) < 1,\\
		(-\infty,0], & u_i(t)=1,\\
		\{-\infty\}, &u_i>1.
	\end{cases}\]
Notice that this expression only makes sense for trajectories such that $u(t)\in\V_{[0,1]}$ at a.e. $t$, as $\partial I_{[0,1]}(x)$ has no real values for $x\notin[0,1]$. For tidyness of notation, we define 
\begin{equation}
\label{beta}
	 \mathcal{B}(u) :=\left\{  \alpha\in\V\: \middle|\: \forall i\in V: \alpha_i\in -\partial I_{[0,1]}(u_i)
	\right\}
\end{equation}
which is non-empty if and only if $u\in\V_{[0,1]}$. 

We can in fact characterise $\beta$ more exactly, as a.e. an explicit function of $u$, so the AC flow remains a differential equation rather than a differential inclusion. We first note a standard fact about continuous representatives of $H^1$ functions.
\begin{lem}\label{H1AClem}
If $u\in H^1_{loc}(T;\V)\cap C^0(T;\V)$, then $u$ is locally absolutely continous on $T$. Hence $u$ is differentiable with weak derivative equal to the classical derivative a.e. in $T$.
\end{lem} 
\begin{proof} By Proposition \ref{graphH1}, $u\in H^1_{loc}(T;\V)\cap C^0(T;\V)$ if and only if for all $i\in V$, $u_i\in H^1_{loc}(T;\mathbb{R})\cap C^0(T;\mathbb{R})$. The result then follows from standard results, cf. \cite[Theorem 7.13]{Leoni}.
\end{proof}
\begin{thm}\label{betathm}
Let $(u,\beta)$ obey \eqref{ACobs1} at a.e. $t\in T$, with $u\in H^1_{loc}(T;\V)\cap C^0(T;\V)\cap\V_{[0,1],t\in T}$. Then for all $i\in V$ and a.e. $t\in T$, \be\label{beta2} \beta_i(t) = \begin{cases} \frac{1}{2} +\varepsilon(\Delta u(t))_i, &u_i(t)=0,\\
0, & u_i(t)\in(0,1),\\
-\frac{1}{2} +\varepsilon(\Delta u(t))_i, & u_i(t) =1.\end{cases}\ee
\end{thm}
\begin{proof}Since $\beta(t)\in\mathcal{B}(u(t))$ for a.e. $t\in T$, \eqref{beta2} holds at a.e. $t\in T$ and $i\in V$ such that $u_i(t) \in(0,1)$. Let $\tilde T\subseteq T$ denote the times when $u$ is differentiable with classical derivative equal to its weak derivative.  
Since $u_i(t)\in[0,1]$ at all times, when $t\in \tilde T$ and $u_i(t) \in\{ 0,1\}$ we have $du_i/dt = 0$. 
Consider first $u_i(t) = 0$. Then for a.e. such $t\in \tilde T$ \[0= \frac{du_i}{dt}= -(\Delta u(t))_i + \frac{1}{\varepsilon}\left( \beta_i(t)-\frac{1}{2}\right ) \] 
and therefore
\[\beta_i(t) = \frac{1}{2} +\varepsilon(\Delta u(t))_i\]
and likewise for $u_i(t)=1$. Thus \eqref{beta2} holds at a.e. $t\in \tilde T$. By Lemma \ref{H1AClem}, $T\setminus\tilde T$ has measure zero, so  \eqref{beta2} holds at a.e. $t\in T$. 
\end{proof}
\begin{nb}
From \eqref{beta2} and the sign properties of $(\Delta u(t))_i$ at $u_i(t)\in\{0,1\}$, it follows that $\beta(t)\in\V_{[-1/2,1/2]}$. This corresponds to the explicit restriction on the subdifferential imposed in the definition of $\beta$ in \emph{\cite{BE1993}}.  
\end{nb}
%Following \cite{BE1993}, we therefore note that our solutions to \eqref{ACobs1} have the subdifferential term $\beta(t)$ almost always lying in the set [[REWRITE]]
%\begin{equation}
%\label{beta}
%	 \mathcal{B}(u(t)) :=\left\{  \alpha\in\V\: \middle|\: \forall i\in V: \alpha_i\in -\partial I_{[0,1]}(u_i(t)), \:|\alpha_i|\leq \frac{1}{2}\text{ when }u_i(t)\in\{0,1\}
%	\right\}.
%\end{equation}
%Unpacking the definition, this means $\mathcal{B}(u(t))$ is nonempty only for $u(t)\in\V_{[0,1]}$ and 
%\[ \beta_i(t)\in\begin{cases}
%		%\{\infty\}, & u_i<0,\\
%		[0,1/2], & u_i(t)=0,\\
%		\{0\}, &0<u_i(t) < 1,\\
%		[-1/2,0], & u_i(t)=1.
%		%\{-\infty\}, &u_i>1.
%	\end{cases}\]
%Tying this all together, we define AC solutions as follows:
In light of the above results, we rigorously define the double-obstacle AC flow as follows.
\begin{mydef}[Double-obstacle AC flow]\label{ACdef}
Let $T$ be an interval. Then the pair $(u,\beta)\in\V_{[0,1],t\in T}\times\V_{t\in T}$ is a solution to double-obstacle AC flow on $T$ when $u\in H_{loc}^1(T;\V)\cap C^0(T;\V)$ and for almost every $t\in T$, 
\begin{align}
\label{ACobs2}
	&\varepsilon \frac{du}{dt} + \varepsilon\Delta u(t) +\frac{1}{2}\mathbf{1}-u(t)= \beta(t), &\beta(t)\in\mathcal{B}(u(t)).
\end{align}
Note that we will often for conciseness refer to just $u$ as a solution to \eqref{ACobs2}, since $\beta$ is a.e. uniquely determined as a function of $u$ by \eqref{beta2}.
\end{mydef}
\begin{nb}
The condition that $u\in C^0$ is just to pick out the \emph{continuous representative} of $u\in H^1_{loc}$. Recall since $T$ is one-dimensional we have by Sobolev embedding that any $u\in H^1(T;\V)$ has a representative $\tilde u \in C^{0,1/2}(T;\V)$ such that $u(t) = \tilde u(t)$ for a.e. $t\in T$.
\end{nb}
\subsection{Alternative forms of the double-obstacle AC flow}
We here give an explicit integral form and, following \cite{BE1993}, a weak form of the AC flow.
\begin{thm}\label{ACexplicit}
The pair $(u,\beta)\in\V_{[0,1],t\in T}\times\V_{t\in T}$ is a solution to \eqref{ACobs2} if and only if $\beta$ is locally integrable and locally bounded a.e. \emph{(}i.e. integrable and bounded a.e. on every bounded subset of $T$\emph{)}, $\beta(t)\in\mathcal{B}(u(t))$ for a.e. $t\in T$, and for all $t\in T$
\begin{equation}
	\label{ACobssoln2}
	u(t) = \frac{1}{2}\mathbf{1} +  e^{t/\varepsilon}e^{-t\Delta}\left(u(0)-\frac{1}{2}\mathbf{1}\right) + \frac{1}{\varepsilon}\int_0^t e^{(t-s)/\varepsilon}e^{-(t-s)\Delta}\beta(s) \; ds.
\end{equation}
\end{thm}
\begin{proof}
Let $(u,\beta)$ solve \eqref{ACobs2}. Recall from Theorem \ref{betathm} that $\beta\in\V_{[-1/2,1/2],t\in T}$ a.e., so is locally (indeed, globally) bounded a.e. Next, note from \eqref{ACobs2} that $\beta$ is a sum of a continuous function and the derivative of an $H^1_{loc}$ function, so is locally integrable. Finally, note that we can rewrite the ODE in \eqref{ACobs2} as \[\frac{d}{dt}\left(e^{-t/\varepsilon}e^{t\Delta}\left(u-\frac{1}{2}\mathbf{1}\right)\right)=\varepsilon^{-1}e^{-t/\varepsilon}e^{t\Delta}\beta.\] Then $u(t)$ obeys  \eqref{ACobssoln2} by the `fundamental theorem of calculus' on $H^1$ \cite[Theorem 8.2]{Brezis} since $e^{-s/\varepsilon}e^{s\Delta}(u-\frac{1}{2}\mathbf{1})\in H^1((0,t);\V)$.

Now let $\beta$ be locally integrable and locally bounded a.e., and let $u(t)$ obey \eqref{ACobssoln2} and $\beta(t)\in\mathcal{B}(u(t))$ for a.e. $t\in T$. Then $u$ has weak derivative  
\[
\frac{du}{dt} = \left(\frac{1}{\varepsilon}I -\Delta\right)e^{t/\varepsilon}e^{-t\Delta}\left(u(0)-\frac{1}{2}\mathbf{1}\right) + \frac{1}{\varepsilon}\beta(t) + \left(\frac{1}{\varepsilon}I -\Delta\right) \frac{1}{\varepsilon}\int_0^t e^{(t-s)/\varepsilon}e^{-(t-s)\Delta}\beta(s) \; ds
\]
which obeys \eqref{ACobs2}. Next, since $\beta$ is locally bounded a.e. it follows from \eqref{ACobssoln2} that $u$ is continuous. Finally $u$ is bounded, so is locally $L^2$, and by above $du/dt$ is a sum of (respectively) a smooth function, a locally bounded a.e. function and the integral of a locally bounded a.e. function, so is locally bounded a.e. and hence is locally $L^2$. Hence $u\in H^1_{loc}(T;\V)$. 
\end{proof}
\begin{nb} This form will be important later, when considering the convergence of the semi-discrete scheme. Note that \eqref{ACobssoln2} implies $\beta(s)\neq \mathbf{0}$ a positive measure subset of the time, so $u$ must remain at obstacles non-instantaneously.
\end{nb}
\begin{thm}\label{ACobsweak}
	A continuous function $u\in\V_{[0,1],t\in T}\cap H^1_{loc}(T;\V)$ \emph{(}and associated $\beta(t)= \varepsilon \frac{du}{dt} + \varepsilon\Delta u(t)-u(t)+\frac{1}{2}\mathbf{1}$ a.e.\emph{)} is a solution to \eqref{ACobs2} if and only if for almost every $t\in T$
		\begin{equation}
		\label{ACobs2a}
		\forall\eta\in\V_{[0,1]},\: \left\langle\varepsilon\frac{du}{dt}-u(t)+\frac{1}{2}\mathbf{1},\eta-u(t)\right\rangle_\V+\varepsilon\left\langle\nabla u(t),\nabla\eta-\nabla u(t)\right\rangle_\mathcal{E}\ge0.
	\end{equation}
\end{thm}
\begin{proof}
	Let $u$ solve \eqref{ACobs2} so for a.e. $t\in T$, $\beta(t)\in\mathcal{B}(u(t))$. Then at a.e. $t\in T$, $\beta_i(t) \geq 0$ at the $i\in V$ where $u_i(t) =0$ and $\beta_i(t)\leq 0$ at the $i\in V$ where $u_i(t) =1$. So, at each such $t\in T$, %for any $\eta\in\V_{[0,1]}$ 
\begin{equation*}
\forall \eta\in\V_{[0,1]}  \left\{ %\qquad
\begin{aligned}
	 LHS \: \eqref{ACobs2a} &= \left\langle-\varepsilon\Delta u(t)+\beta(t),\eta-u(t)\right\rangle_\V+\varepsilon\left\langle\nabla u(t),\nabla\eta-\nabla u(t)\right\rangle_\mathcal{E}\\
	 &=\ip{\beta(t),\eta-u(t)} \\
	 &= \sum_{\{i|u_i(t)=0\}} d_i^r\beta_i(t)\eta_i + \sum_{\{i|u_i(t)=1\}} d_i^r\beta_i(t)(\eta_i-1)\geq 0.
\end{aligned}\right.
\end{equation*}
	Now let $u\in\V_{[0,1],t\in T}\cap H^1_{loc}(T;\V)$ satisfy \eqref{ACobs2a} for almost every $t\in T$.	 So for all such $t$ \[ \forall\eta\in\V_{[0,1]},\:\ip{\beta(t),\eta-u(t)}\ge 0. \] 
Let $\eta_j =u_j(t)$ for $j\neq i$ and $\eta_i=0$, and  $\eta'_j =u_j(t)$ for $j\neq i$ and $\eta'_i=1$. Substituting $\eta$ and $\eta'$ into the above we have $\beta_i(t) u_i(t)\leq 0$ and $\beta_i(t)(1-u_i(t))\geq 0$. Therefore
	\[\beta_i(t)\begin{cases}
		=0, & u_i(t)\in (0,1)\\
		\leq 0, &u_i(t)=1\\
		\geq 0, &u_i(t)=0
	\end{cases}\]
so $\beta(t)\in\mathcal{B}(u(t))$, and thus $(u,\beta)$ solves \eqref{ACobs2}.
\end{proof}
\subsection{Well-posedness of the double-obstacle AC flow}
%We note the following important theorems.
\begin{thm}\label{existence}
Let $T=[0,\infty)$. Then for all  $u_0\in \V_{[0,1]}$ there exists $(u,\beta)\in\V_{[0,1],t\in T}\times\V_{t\in T}$ satisfying \eqref{ACobs2} with $u\in H_{loc}^1(T;\V)\cap C^{0,1}(T;\V)$ and with initial condition $ u(0) = u_0$. 
\end{thm}
\begin{proof}[Proof \emph{(}see Appendix \ref{existsec}\emph{)}] We construct a solution by approximating double-obstacle AC solutions by solutions to AC flows with a $C^1$ approximation of the double-obstacle potential, and taking the limit as the approximations become more accurate. We also prove this as Theorem~\ref{SDlimit}, by taking a limit of the semi-discrete approximations as defined in \eqref{SDobs}.
\end{proof}
\begin{thm}\label{uniquenessprequel} Let $T = [0,T_0]$ or $[0,\infty)$, and let $(u,\beta),(v,\gamma)$ be solutions to \eqref{ACobs2} on T with $u(0)=v(0)$. Then for all $t\in T$, $u(t) = v(t)$, and for a.e. $t\in T$, $\beta(t) = \gamma(t)$.   \end{thm}
\begin{proof}[Proof \emph{(}see Appendix \ref{uniquesec}\emph{)}] 
We prove this as Corollaries~\ref{ACuniqueness} and \ref{ACuniqueness2} by translating into the graph setting a comparision principle derived in Chen and Elliott \cite{CE}.
\end{proof}
\begin{thm}\label{wellp} Let $u,v$ be solutions to \eqref{ACobs2} on $T = [0,T_0]$ or $[0,\infty)$. Then for all $t\in T$
\be\label{wpeq}
|| u(t)- v(t)||_\V \leq e^{t/\varepsilon}||u(0)-v(0)||_\V.
\ee
\end{thm}
\begin{proof}
We prove this in section \ref{consequences} as a consequence of Theorem \ref{SDLips} and Theorem \ref{SDlimit}. Avoiding circularity, this result is not used in proving any results in this paper.
\end{proof}
\subsection{Gradient flow property and regularity of solutions}
%We now investigate important properties of this equation. Firstly, we note the following existence and uniqueness theory, which we shall prove over the course of this paper.
%\begin{nb}
%The conditions in \eqref{beta} naturally drop out of the construction of the solution as in Appendix A as a limit of solutions to AC with $C^1$ approximation of the double-obstacle potential. Likewise for the consruction as in Theorem \ref{SDlimit}.
%\end{nb}
To meaningfully call \eqref{ACobs2} an AC flow, we must verify that it monotonically decreases the Ginzburg--Landau functional.
\begin{prop} \label{ACgf} Let $(u,\beta)$ solve \eqref{ACobs2} on an interval $T$. Then for a.e. $t\in T$, 
\[ \frac{d\GL(u(t))}{dt}\leq 0. \] 
\end{prop}
\begin{proof}
Define 
\[G_\varepsilon(u):=\frac{1}{2}\left|\left|\nabla u(t)\right|\right|_\mathcal{E}^2 +\frac{1}{2\varepsilon}\left\langle u(t),\mathbf{1}-u(t) \right \rangle_\V\]
then by \eqref{GL} we have, for all $t\in T$,
\[ \GL(u(t)) = G_\varepsilon(u(t)) +  \frac{1}{\varepsilon}\left\langle I_{[0,1]}\circ u(t),\mathbf{1}\right \rangle_\V = G_\varepsilon(u(t))  \] 
since $u(t)\in \V_{[0,1]}$ for all $t\in T$. Hence, since $\nabla_\V G_\varepsilon(u) = \Delta u +\frac{1}{\varepsilon}\left(\frac{1}{2}\mathbf{1} -u\right)$, we note that 
\begin{align*}\varepsilon^2\frac{d\GL(u(t))}{dt} &= \varepsilon^2\frac{dG_\varepsilon(u(t))}{dt} = \left\langle\varepsilon \frac{du}{dt},\varepsilon\Delta u(t) +\frac{1}{2}\mathbf{1} -u(t)\right \rangle_\V\\ &= \left\langle \beta(t)-\varepsilon\Delta u(t) -\frac{1}{2}\mathbf{1} +u(t) ,\varepsilon\Delta u(t) +\frac{1}{2}\mathbf{1} -u(t)\right \rangle_\V.\end{align*}
By Theorem \ref{betathm}, at a.e. $t\in T$ and all $i\in V$, if $u_i(t) \in \{0,1\}$, then $\beta_i(t)-\varepsilon(\Delta u(t))_i -\frac{1}{2}+u_i(t)=0$, and if $u_i(t) \in (0,1)$, then $\beta_i(t) = 0$. Thus for a.e. $t\in T$, let $V'_t:=\{i\in V\mid u_i(t)\in(0,1)\}$, then
\begin{align*}&\left\langle \beta(t)-\varepsilon\Delta u(t) -\frac{1}{2}\mathbf{1} +u(t) ,\varepsilon\Delta u(t) +\frac{1}{2}\mathbf{1} -u(t)\right \rangle_\V%\\
=  -\sum_{i\in V'_t} d_i^r\left(\varepsilon(\Delta u(t))_i +\frac{1}{2}-u_i(t)\right)^2\leq 0
\end{align*}
as desired.
\end{proof}
Finally,  we use \eqref{ACobssoln2} to derive a regularity estimate.
\begin{thm}\label{ACLips}
AC solutions are globally Lipschitz, i.e. $ u \in C^{0,1}([0,\infty);\V)$.
\end{thm}
\begin{proof}
By \eqref{ACobssoln2} we have for $t_1<t_2$ and $A:= \varepsilon^{-1}I - \Delta$
\begin{equation*}\begin{split}  &u(t_2)- u(t_1) \\
&= \left(e^{t_2 A} -e^{t_1 A}\right)\left(u(0) -\frac{1}{2}\mathbf{1}\right)+\frac{1}{\varepsilon} \int_0^{t_1}  \left( e^{(t_2-s)A} - e^{(t_1-s)A}\right)\beta(s)\; ds+\frac{1}{\varepsilon}\int_{t_1}^{t_2}  e^{(t_2-s)A}\beta(s) \; ds \\
&=\left(e^{(t_2-t_1)A}-I\right)\left[e^{t_1A}\left(u(0) -\frac{1}{2}\mathbf{1}\right)+\frac{1}{\varepsilon} \int_0^{t_1}  e^{(t_1-s)A}\beta(s)\; ds\right] +\frac{1}{\varepsilon}\int_{0}^{t_2-t_1}  e^{sA}\beta(t_2-s)\; ds\\
&=\left(e^{(t_2-t_1)A}-I\right)\left( u(t_1)-\frac{1}{2}\mathbf{1}\right) +\frac{1}{\varepsilon}\int_{0}^{t_2-t_1}  e^{sA}\beta(t_2-s)\; ds.
\end{split} \end{equation*}
Recall that $ u(t) \in \V_{[0,1]}$ for all $t\geq 0$, and by Theorem~\ref{betathm} that $\beta(t)\in\V_{[-1/2,1/2]}$ for a.e. $t\geq0$. Let $B_{\delta t}: = (e^{\delta t A}- I)/\delta t $. Then since $A$ has largest positive eigenvalue $1/\varepsilon$ and $B_{\delta t}$ is self-adjoint, \[||B_{\delta t}|| = \delta t^{-1} \left(e^{\delta t /\varepsilon}- 1\right) \] and note this is monotonically increasing in $\delta t$ for $\delta t>0$. We thus have for $0<t_2 -t_1 < 1$ 
 \begin{equation*}\begin{split}\frac{\left|\left| u(t_2)- u(t_1)\right|\right|_\V}{t_2-t_1}&\leq ||B_{t_2-t_1}||\cdot \frac{1}{2}||\mathbf{1}||_\V +\frac{1}{\varepsilon} \underset{s\in [0,t_2-t_1]}{\text{ess sup}} \left|\left|e^{sA}\beta(t_2-s)  \right|\right|_\V \\
 &\leq  \frac{e^{(t_2-t_1)/\varepsilon}- 1}{ t_2-t_1 }\cdot  \frac{1}{2}||\mathbf{1}||_\V + \frac{1}{\varepsilon}\sup_{s\in [0,t_2-t_1]} \left|\left|e^{sA}\right|\right| \cdot \frac{1}{2}||\mathbf{1}||_\V \\
&\leq \frac{e^{(t_2-t_1)/\varepsilon}- 1}{ t_2-t_1 }\cdot  \frac{1}{2}||\mathbf{1}||_\V + \frac{1}{\varepsilon} e^{(t_2-t_1)/\varepsilon}  \cdot \frac{1}{2}||\mathbf{1}||_\V \\
&< \frac{1}{2}||\mathbf{1}||_\V \left (e^{1/\varepsilon}  -1 + \frac{1}{\varepsilon}e^{1/\varepsilon}\right) 
\end{split}\end{equation*} and for $t_2 -t_1\geq 1$ we have the simpler estimate 
\[\frac{\left|\left| u(t_2)- u(t_1)\right|\right|_\V}{t_2-t_1}\leq \left|\left| u(t_2)- u(t_1)\right|\right|_\V \leq ||\mathbf{1}||_\V \] completing the proof.
\end{proof} 
\begin{nb}
This regularity estimate is relatively optimal, i.e. $u$ is not in general $C^1$. For example, suppose $u(0) = \alpha \mathbf{1}$ for $\alpha \in (0,1/2)$, and take as an ansatz for \eqref{ACobs2}{:}
\begin{align*}
&u(t) = \begin{cases}\frac{1}{2}\mathbf{1} +\left(\alpha-\frac{1}{2}\right)e^{t/\varepsilon}\mathbf{1}, &0\leq t < -\varepsilon\log(1-2\alpha),\\
\mathbf{0}, & t\geq  -\varepsilon\log(1-2\alpha),\end{cases} &\beta(t) = \begin{cases}\mathbf{0}, &0\leq t < -\varepsilon\log(1-2\alpha),\\
\frac{1}{2}\mathbf{1}, & t\geq  -\varepsilon\log(1-2\alpha).\end{cases}
\end{align*}
One can check that this solves \eqref{ACobs2} and is not differentiable at $t = -\varepsilon\log(1-2\alpha)$.
\end{nb}
%Finally, we return to verifying the gradient flow property:
%\begin{thm}\label{betathm}
%AC solutions $(u,\beta)$ satisfy \eqref{GFcondition} for almost every $t$. Indeed, for almost every $t$ such that $u_i(t)\in\{0,1\}$, we have \be\label{beta2} |\beta_i(t)| = \frac{1}{2} -\varepsilon|(\Delta u(t))_i|.\ee
%We recall that this entails that $\GL(u(t))$ is monotonically decreasing in $t$. 
%\end{thm}
%\begin{proof}
%Since $u_i(t)\in[0,1]$ at all times, when $u_i$ is differentiable at $t$ and $u_i(t) \in\{ 0,1\}$ we have $du_i/dt = 0$. 
%Consider first $u_i(t) = 0$. Then we have $(\Delta u(t))_i\leq 0$ and for a.e. $t$ \[0= \frac{du_i}{dt}(t) = -(\Delta u(t))_i + \frac{1}{\varepsilon}\left( \beta_i(t)-\frac{1}{2}\right ) \] 
%therefore since $\beta_i(t)\geq 0$ 
%\[|\beta_i(t)|=\beta_i(t) = \frac{1}{2} -\varepsilon|(\Delta u(t))_i|.\]
%We argue likewise for $u_i(t)=1$. Thus \eqref{beta2} holds at a.e. $t$ such that $u_i$ is differentiable at $t$ and $u_i(t) \in\{ 0,1\}$. Finally by the above theorem $u$ is Lipschitz and is therefore differentiable a.e., completing the proof. 
%\end{proof}
\section{Semi-discrete scheme}\label{SDsec}
As in the previous section, we must extend \eqref{sdAC} to the non-differentiable case of the double-obstacle potential. Then our semi-discrete scheme becomes 
\be\label{SDobs0}
\frac{u_{n+1} - e^{-\tau\Delta}u_n}{\tau}\in - \frac{1}{\varepsilon}\partial W(u_{n+1}).\ee
More explicitly, we make the following definition.
\begin{mydef}[Semi-discrete scheme]\label{SDdef}
Given $u_0\in\V_{[0,1]}$ and writing $\lambda:=\tau/\varepsilon$, we define a \emph{semi-discrete scheme} $(u_n,\beta_n)$ with initial condition $u_0$, for all $n\in\mathbb{N}$, by the rule\emph{:}
\begin{align}
\label{SDobs} &(1-\lambda)u_{n+1} -e^{-\tau\Delta}u_n+\frac{\lambda}{2}\mathbf{1} =\lambda\beta_{n+1}, & \beta_{n+1}\in\mathcal{B}(u_{n+1}),
\end{align}
recalling the notation $\mathcal{B}(u_{n+1})$ from \eqref{beta}. Note that for given $u_n$, $u_{n+1}$ satisfies \eqref{SDobs0} if and only if $(u_{n+1},\beta_{n+1})$ satisfy \eqref{SDobs} for some $\beta_{n+1}\in\mathcal{B}(u_{n+1})$.
\end{mydef} 
\subsection{Variational form and link to the MBO scheme}
We now prove the central result of this paper.
\begin{thm}
\label{obsMMthm}
If $0\leq\tau\leq\varepsilon$ then $(u_{n+1},\beta_{n+1})$ is a solution to the semi-discrete scheme \eqref{SDobs} for some $\beta_{n+1}\in\mathcal{B}(u_{n+1})$ if and only if $u_{n+1}$ solves the variational equation\emph{:}
 \begin{equation}
	\label{ACobsMM}
	u_{n+1}\in \underset{u\in \V_{[0,1]}}{ \argmin }  \: \lambda\left\langle u,\mathbf{1}-u\right \rangle_\V +  \left|\left|u-e^{-\tau\Delta}u_n\right|\right|^2_\V.
\end{equation}
Furthermore if $\tau = \varepsilon$ then \eqref{ACobsMM} is equivalent to \eqref{MBOMM}, the minimisation governing the MBO updates, so in this case the semi-discrete scheme is identical to the MBO scheme. 

Finally if $\tau < \varepsilon$ then \eqref{ACobsMM} has unique solution
\begin{equation}
	\label{ACobssoln}
	(u_{n+1})_i=\begin{cases}
		0, &\text{ if } \left(e^{-\tau\Delta}u_n\right)_i < \frac{1}{2}\lambda
\\
		\frac{1}{2} + \frac{\left(e^{-\tau\Delta}u_n\right)_i - 1/2}{1-\lambda}
, &\text{ if }\frac{1}{2}\lambda
\leq\left(e^{-\tau\Delta}u_n\right)_i < 1-\frac{1}{2}\lambda
\\
		1, &\text{ if } \left(e^{-\tau\Delta}u_n\right)_i \geq 1-\frac{1}{2}\lambda
	\end{cases}
\end{equation}
with corresponding $\beta$ term
\be\label{betasoln}
(\beta_{n+1})_i =\begin{cases}
		\frac{1}{2}-\lambda^{-1}(e^{-\tau\Delta}u_n)_i, &\text{ if } \left(e^{-\tau\Delta}u_n\right)_i < \frac{1}{2}\lambda, \\
		0, &\text{ if }\frac{1}{2}\lambda
\leq\left(e^{-\tau\Delta}u_n\right)_i < 1-\frac{1}{2}\lambda,
\\
		-\frac{1}{2}+\lambda^{-1}(1-(e^{-\tau\Delta}u_n)_i), &\text{ if } \left(e^{-\tau\Delta}u_n\right)_i \geq 1-\frac{1}{2}\lambda.
	\end{cases}\ee
\end{thm}
\begin{proof}
Let $(u_{n+1},\beta_{n+1})$ solve \eqref{SDobs}. First, note that $\mathcal{B}(u_{n+1})$ is non-empty and so $u_{n+1}\in\V_{[0,1]}$. We seek to prove that for $0\leq\lambda\leq 1$ and $\forall\eta\in\V_{[0,1]}$:
\begin{equation*}\begin{split}
&\lambda\ip{u_{n+1},\mathbf{1}-u_{n+1}}+ \left\langle u_{n+1}-e^{-\tau\Delta}u_n,u_{n+1}-e^{-\tau\Delta}u_n\right \rangle_\V
\\& 
\leq \lambda\ip{\eta,\mathbf{1}-\eta}+ \left\langle \eta-e^{-\tau\Delta}u_n,\eta-e^{-\tau\Delta}u_n\right \rangle_\V 
\end{split}\end{equation*}
By rearranging and cancelling this is equivalent to \begin{equation*} \begin{split}
	0& \leq  \left\langle \eta-u_{n+1},\lambda\mathbf{1}-2 e^{-\tau\Delta}u_n \right \rangle_\V + (1-\lambda)\left(\ip{\eta,\eta}-\ip{u_{n+1},u_{n+1}}\right)\\
	&= \left\langle \eta-u_{n+1},\lambda\mathbf{1}-2 e^{-\tau\Delta}u_n +(1-\lambda)(\eta+u_{n+1}) \right \rangle_\V\\
	&= \left\langle \eta-u_{n+1},2\lambda\beta_{n+1} +(1-\lambda)(\eta-u_{n+1}) \right \rangle_\V \\
	&= 2\lambda\left\langle \eta-u_{n+1},\beta_{n+1}\right \rangle_\V +(1-\lambda)||\eta-u_{n+1}||^2_\V 
\end{split}  \end{equation*}
but it is easy to check from \eqref{beta} that $(\beta_{n+1})_i$ is either zero, when $(u_{n+1})_i\in(0,1)$, or has the same sign as $\eta_i -(u_{n+1})_i$ since $\eta_i\in [0,1]$, so $\left\langle \eta-u_{n+1},\beta_{n+1}\right \rangle_\V\geq 0$.

Now let $u$ solve \eqref{ACobsMM}. The functional in \eqref{ACobsMM} can be written \[ \lambda\left\langle u,\mathbf{1}-u\right \rangle_\V +  \left|\left|u-e^{-\tau\Delta}u_n\right|\right|^2_\V = \sum_{i\in V} d_i^r \left(\lambda u_i(1-u_i) + \left(u_i-\left(e^{-\tau\Delta}u_n\right)_i\right)^2\right).\] 
%so we can reduce 	\eqref{ACobsMM} to the system of 1-dimensional problems \[ (u_{n+1})_i \in \underset{x\in [0,1]}{ \argmin } \: \lambda x(1-x) + \left(x-\left(e^{-\tau\Delta}u_n\right)_i\right)^2.\]
Differentiating, we get that, for $0\leq\lambda < 1$, $\lambda x(1-x) + \left(x-\left(e^{-\tau\Delta}u_n\right)_i\right)^2$ is minimised at  
\[
	x = \frac{\left(e^{-\tau\Delta}u_n\right)_i - \lambda/2}{1-\lambda} = \frac{1}{2} + \frac{\left(e^{-\tau\Delta}u_n\right)_i - 1/2}{1-\lambda}.
\]
Therefore for $0\leq \lambda<1$ the solution $u$ is given by 
\begin{equation*}
	u_i=\begin{cases}
		0, &\text{ if } \left(e^{-\tau\Delta}u_n\right)_i < \frac{1}{2}\lambda
\\
		\frac{1}{2} + \frac{\left(e^{-\tau\Delta}u_n\right)_i - 1/2}{1-\lambda}
, &\text{ if }\frac{1}{2}\lambda
\leq\left(e^{-\tau\Delta}u_n\right)_i < 1-\frac{1}{2}\lambda
\\
		1, &\text{ if } \left(e^{-\tau\Delta}u_n\right)_i \geq 1-\frac{1}{2}\lambda
	\end{cases}
\end{equation*}
and hence 
\[\frac{1}{\lambda}\left((1-\lambda)u_i -(e^{-\tau\Delta}u_n)_i+\frac{\lambda}{2}\right) =\begin{cases}
		\frac{1}{2}-\frac{1}{\lambda}(e^{-\tau\Delta}u_n)_i, &\text{ if } \left(e^{-\tau\Delta}u_n\right)_i < \frac{1}{2}\lambda, \\
		0, &\text{ if }\frac{1}{2}\lambda
\leq\left(e^{-\tau\Delta}u_n\right)_i < 1-\frac{1}{2}\lambda,
\\
		-\frac{1}{2}+\frac{1}{\lambda}(1-(e^{-\tau\Delta}u_n)_i), &\text{ if } \left(e^{-\tau\Delta}u_n\right)_i \geq 1-\frac{1}{2}\lambda.
	\end{cases}\]
Thus $\beta = \lambda^{-1}\left((1-\lambda)u -e^{-\tau\Delta}u_n+\frac{\lambda}{2}\mathbf{1}\right) \in\mathcal{B}(u)$, so $u$ solves \eqref{SDobs}.

If $\lambda=1$ then examine the functional in \eqref{ACobsMM} for $\lambda = 1$:
\begin{equation}\label{lamb1}\begin{split}
%\hspace*{-0.2em}
&\left\langle u,\mathbf{1}-u\right \rangle_\V +  \left|\left|u-e^{-\tau\Delta}u_n\right|\right|^2_\V
\\&
= \ip{u,\mathbf{1}} -\ip{u,u} + \ip{u,u} - 2 \left\langle u,e^{-\tau\Delta}u_n\right \rangle_\V +  \left\langle e^{-\tau\Delta}u_n,e^{-\tau\Delta}u_n\right \rangle_\V\\
&\simeq \left\langle u,\mathbf{1}-2e^{-\tau\Delta}u_n\right \rangle_\V, 
\end{split}\end{equation}
and therefore $u$ as a minimiser must obey \[u_i \in \begin{cases}
	\{1\}, &(e^{-\tau\Delta}u_n)_i > 1/2,\\
	[0,1], &(e^{-\tau\Delta}u_n)_i = 1/2,\\
	\{0\}, &(e^{-\tau\Delta}u_n)_i < 1/2.
\end{cases} \]
Hence $\beta\in\mathcal{B}(u)$ if and only if for each $i\in V$ 
\[
\beta_i\in\begin{cases}
		[0,\infty), & (e^{-\tau\Delta}u_n)_i \leq 1/2\\
		\{0\}, &(e^{-\tau\Delta}u_n)_i = 1/2, u_i\in(0,1)\\
		(-\infty,0], & (e^{-\tau\Delta}u_n)_i \geq 1/2
	\end{cases}
\]
and thus $\frac{1}{2}\mathbf{1}- e^{-\tau\Delta}u_n\in\mathcal{B}(u)$, so $u$ solves \eqref{SDobs}.
\end{proof}
%\begin{proof}[Proof of Theorem \ref{thm1}]
%Follows directly from \eqref{lamb1}: for $\lambda=1$ the functional in \eqref{ACobsMM} is the functional describing MBO.
%\end{proof} 
\begin{nb}
The $0\leq\lambda<1$ semi-discrete solution tends to the MBO solution as $\lambda\uparrow 1$. Indeed, it is a  relaxation of the MBO thresholding \emph{(}see Fig. \ref{SDfig}\emph{)}.
\end{nb}
\begin{figure}[H] 
\centering
\begin{tikzpicture}
\begin{axis}[
	%scale only axis,
	height = 6cm,
    axis lines = left,axis line style={-},
    xlabel = $\left(e^{-\tau\Delta}u_n\right)_i$,
    ylabel = {$(u_{n+1})_i$},
    xtick={0,0.5,1},
    ytick={0,0.5,1},
    legend pos=south east,
    extra x ticks={0.3,0.7},
    extra x tick style={
    tick label style={
    ,anchor=north}},
    extra x tick labels={$\frac{1}{2}\lambda$,$1-\frac{1}{2}\lambda$},
xmin=0,xmax=1
]

\addplot [
    domain=0:0.3, 
    samples=100, 
    color=red,
]
{0};

\addplot [
    domain=0.3:0.7, 
    samples=100, 
    color=red,
]
{0.5 + (x-0.5)/0.4};
\addplot [
    domain=0.7:1, 
    samples=100, 
    color=red,
]
{1};
\end{axis}
\begin{axis}[
%scale only axis,
height = 6cm,
axis y line*=right,
axis x line=none,
	axis line style={-},
    ylabel = {$(\beta_{n+1})_i$},
    ytick={-0.5,0,0.5},
xmin=0,xmax=1,ymin=-0.5,ymax=0.5]

\addplot [
    domain=0:0.3, 
    samples=100, 
    color=blue,
]
{0.5-x/0.6};

\addplot [
    domain=0.3:0.7, 
    samples=100, 
    color=blue,
]
{0};
\addplot [
    domain=0.7:1, 
    samples=100, 
    color=blue,
]
{-0.5+(1-x)/0.6};
\end{axis}
\end{tikzpicture}
\caption{Plot of the semi-discrete updates $u_{n+1}$ (red, left axis, see \eqref{ACobssoln}) and $\beta_{n+1}$ (blue, right axis, see \eqref{betasoln}),  for $\lambda \in[0,1)$, at $i\in V$ against the diffused value at $i$. Observe that the $u_{n+1}$ solution is a piecewise linear softening of the MBO thresholding.}
\label{SDfig}
\end{figure}
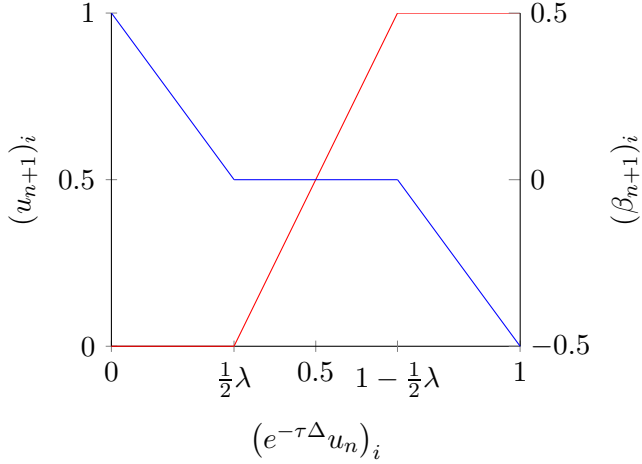 
As a simple consequence, we prove a (Lipschitz) continuity property we will later use when proving well-posedness of double-obstacle AC flow (i.e. Theorem \ref{wellp}).
\begin{thm} \label{SDLips} 
For $\lambda < 1$ and all $n\in\mathbb{N}$, if $u_n$ and $v_n$ are defined according to Definition~\ref{SDdef} with initial states $u_0,v_0\in\V_{[0,1]}$ then
\be\label{SDlipeqn}
||u_n-v_n||_\V \leq (1-\lambda)^{-n}||u_0-v_0||_\V.
\ee
\end{thm}
\begin{proof}
Let $\rho_\lambda:\V_{[0,1]}\rightarrow\V_{[0,1]}$ be the thresholding operator in \eqref{ACobssoln}, i.e. 
\[
(\rho_\lambda(u))_i:=\begin{cases}
		0, &\text{if } u_i < \frac{1}{2}\lambda,
\\
		\frac{1}{2} + \frac{u_i - 1/2}{1-\lambda}
, &\text{if }\frac{1}{2}\lambda
\leq u_i < 1-\frac{1}{2}\lambda,
\\
		1, &\text{if } u_i \geq 1-\frac{1}{2}\lambda.
	\end{cases}
\]
\addtocounter{footnote}{1}
Then by Theorem \ref{obsMMthm} it follows that \footnotetext{We remark the curiosity that, thus expressed, this scheme can be seen to resemble a forward pass through a neural net, with diffusion playing the role of the linear operator and $\rho_\lambda$ the activation function.}
\addtocounter{footnote}{-1}
\begin{align*}
&u_n = (\rho_\lambda \circ e^{-\tau\Delta})^n(u_0), & v_n = (\rho_\lambda \circ e^{-\tau\Delta})^n(v_0) .\footnotemark
\end{align*}
Note that $e^{-\tau\Delta}$ is linear, so it is Lipschitz with constant $||e^{-\tau\Delta}|| = \max \{ e^{-\tau\mu} \mid \mu\in\sigma(\Delta)\} = 1$, and $\rho_\lambda(u)$ is piecewise linear in $u$ with greatest slope $(1-\lambda)^{-1}$ and hence it is Lipschitz with constant $(1-\lambda)^{-1}$. Thus $(\rho_\lambda \circ e^{-\tau\Delta})^n$ is Lipschitz with constant $(1-\lambda)^{-n}$.
\end{proof}
In \cite[Theorem 4.2]{vGGOB}, it was proved that for $\tau$ below certain thresholds and $u_n=\chi_S$, the indicator function of $S\subseteq V$, the MBO scheme exhibits ``pinning" (or ``freezing"). That is, $u_{n+1} = \chi_S$. We here prove an analogous result for our semi-discrete scheme. 
\begin{thm}[Cf. $\text{\cite[Theorem 4.2]{vGGOB}}$]\label{pinningthm}
If $u_n=\chi_S$ and $0\leq\lambda <1$, then the semi-discrete scheme pins, i.e. $u_{n+1}=\chi_S$, if 
\begin{align}
&\label{pin1}\tau\leq||\Delta||^{-1}\log\left(1 +\frac{\lambda}{2} \sqrt{\frac{\min_{i\in V} d_i^r}{\ip{\chi_S,\mathbf{1}}}}\,\right), \text{ or}\\
&\label{pin2}\tau \leq\frac{\lambda}{2||\Delta\chi_S||_\infty}.
\end{align}
\end{thm}
\begin{proof}
The proof of  \cite[Theorem 4.2]{vGGOB} followed from observing that for the MBO scheme, $u_{n+1}=u_n$ if $||e^{-\tau\Delta}u_n-u_n||_\infty<1/2$. Inspecting \eqref{ACobssoln} we see that, for $0\leq \lambda<1$, we likewise have for the semi-discrete scheme that if $u_n=\chi_S$, then $u_{n+1}=\chi_S$ if and only if $ ||e^{-\tau\Delta}\chi_S-\chi_S||_\infty\leq\lambda/2$. Hence, the theorem is proved by modifying appropriately the proof of \cite[Theorem 4.2]{vGGOB} to derive sufficient conditions for this inequality to hold. 
\end{proof}
\begin{nb}Note that for $\lambda=1$, the result of Theorem~\ref{pinningthm} still holds if we turn \eqref{pin1} into a strict inequality, as per \emph{\cite[Theorem 4.2]{vGGOB}}.  \end{nb}
\begin{nb}
Note that \eqref{pin2} is equivalent to $\varepsilon\leq \frac{1}{2} ||\Delta\chi_S||_\infty^{-1}$. Since, as we shall prove in Theorem \ref{SDlimit}, the semi-discrete iterates converge to the AC solution as $\tau\downarrow 0$ for $\varepsilon$ fixed, it follows that if $\varepsilon\leq \frac{1}{2} ||\Delta\chi_S||_\infty^{-1}$ and $u(0) = \chi_S$ then the AC flow also pins, i.e. $u(t) = \chi_S$ for all $t\geq 0$. This can be checked directly from \eqref{ACobs2} by observing that $u(t):= \chi_S$ for all $t\geq 0$ is a valid trajectory \emph{(}i.e. has a corresponding $\beta(t)\in \mathcal{B}(\chi_S)$ for a.e. $t\geq0$\emph{)} if and only if $\varepsilon||\Delta\chi_S||_\infty\leq \frac{1}{2}$, by the characterisation of $\beta$ in Theorem~\ref{betathm}.
\end{nb}
%We round out this subsection by noting some trivia.
%\begin{nb}
%Notice the similarity between \eqref{beta2} and the expression for the $\beta_{n+1}$ term derived in the proof of Theorem~\ref{obsMMthm}{:} 
%%\be
%\begin{subequations}\label{betasoln}
%\begin{align}
%%\begin{split}
%\lambda < 1, &\:\:\: (\beta_{n+1})_i =\begin{cases}
%		\frac{1}{2}-\lambda^{-1}(e^{-\tau\Delta}u_n)_i, &\text{ if } \left(e^{-\tau\Delta}u_n\right)_i < \frac{1}{2}\lambda, \\
%		0, &\text{ if }\frac{1}{2}\lambda
%\leq\left(e^{-\tau\Delta}u_n\right)_i < 1-\frac{1}{2}\lambda,
%\\
%		-\frac{1}{2}+\lambda^{-1}(1-(e^{-\tau\Delta}u_n)_i), &\text{ if } \left(e^{-\tau\Delta}u_n\right)_i \geq 1-\frac{1}{2}\lambda.
%	\end{cases}\\
%\lambda = 1, &\:\:\: \beta_{n+1} = \frac{1}{2}\mathbf{1} - e^{-\tau\Delta} u_n. 
%%\end{split}
%%\ee
%\end{align}\end{subequations}
%Indeed, we can exaggerate this similarity further by noting that $e^{-\tau\Delta} = I - \tau\Delta +\bigO(\tau^2)$ and recalling $\lambda := \tau/\varepsilon$. Therefore for $\tau \ll \varepsilon$, 
%\[
%%\begin{split}
% (\beta_{n+1})_i \approx\begin{cases}
%		\frac{1}{2}+\varepsilon(\Delta u_n)_i-\lambda^{-1}(u_n)_i, &\text{ if } \left(e^{-\tau\Delta}u_n\right)_i < \frac{1}{2}\lambda\\
%		0, &\text{ if }\frac{1}{2}\lambda
%\leq\left(e^{-\tau\Delta}u_n\right)_i < 1-\frac{1}{2}\lambda,
%\\
%		-\frac{1}{2}+\varepsilon(\Delta u_n)_i+\lambda^{-1}(1-(u_n)_i), &\text{ if } \left(e^{-\tau\Delta}u_n\right)_i \geq 1-\frac{1}{2}\lambda
%	\end{cases}
%%\end{split}
%%\ee
%\]
%with $\bigO(\tau)$ error. 
%\end{nb}
\begin{nb} 
We explore the cases when $\lambda\notin[0,1]$. If $\lambda\geq 1$ then $\lambda x(1-x) + \left(x-\left(e^{-\tau\Delta}u_n\right)_i\right)^2$ is concave, so is minimised on the boundary. Thus the $\lambda\ip{u,\mathbf{1}-u}$ term in \eqref{ACobsMM} equals zero, and minimising the norm term gives:
\begin{equation}
	\label{ACobssoln2a}
	(u_{n+1})_i=\begin{cases}
		0, &\text{ if } \left(e^{-\tau\Delta}u_n\right)_i < \frac{1}{2}
\\
		1, &\text{ if } \left(e^{-\tau\Delta}u_n\right)_i > \frac{1}{2}
			\end{cases}
			\end{equation}
and underdetermined if $\left(e^{-\tau\Delta}u_n\right)_i =1/2$. So the variational problem yields the MBO solution even for $\lambda > 1$.  
However, solutions to \eqref{SDobs} are no longer necessarily solutions to \eqref{ACobsMM} for $\lambda > 1$. Let $\lambda = 1 + \delta$, then \eqref{SDobs} becomes \be\label{SDobsge1}-\delta(u_{n+1})_i -(e^{-\tau\Delta}u_n)_i+\frac{1}{2} + \frac{1}{2}\delta =(1+\delta)(\beta_{n+1})_i.\ee
This has admissible solution \emph{(}i.e. there is a corresponsing $\beta_{n+1}\in\mathcal{B}(u_{n+1})$\emph{)} with
\begin{align*}
(u_{n+1})_i  = 0&&\text{if and only if}&&(e^{-\tau\Delta}u_n)_i\in\left[0,\frac{1}{2}(1+\delta)\right],&\\
(u_{n+1})_i =\frac{1}{2}+\frac{1}{2}\delta ^{-1}-\delta ^{-1}(e^{-\tau\Delta}u_n)_i&&\text{if and only if}&&(e^{-\tau\Delta}u_n)_i\in\left[\frac{1}{2}(1-\delta),\frac{1}{2}(1+\delta)\right],&\\
(u_{n+1})_i  = 1&&\text{if and only if}&&(e^{-\tau\Delta}u_n)_i\in\left[\frac{1}{2}(1-\delta),1\right].&
\end{align*}
Hence for $\lambda>1$ if $(e^{-\tau\Delta}u_n)_i\in\left[\frac{1}{2}(1-\delta),\frac{1}{2}(1+\delta)\right]\setminus\left\{\frac{1}{2}\right\}$ for any $i\in V$ then \eqref{SDobs} has solutions that are not solutions to \eqref{ACobsMM}. However, the solutions to \eqref{ACobsMM} remain solutions to \eqref{SDobs}.

Finally we consider $\lambda<0$, though this regime has less obvious meaning. By the same argument as for $0\leq \lambda < 1$ we get that \eqref{ACobsMM} has unique solution \[u := (1-\lambda)^{-1}\left( e^{-\tau\Delta}u_n - \frac{\lambda}{2}\mathbf{1}\right)\in \V_{(0,1)}\] and so \[\beta=\lambda^{-1}\left((1-\lambda)u -e^{-\tau\Delta}u_n+\frac{\lambda}{2}\mathbf{1}\right) = \mathbf{0} \in \mathcal{B}(u)\] since $u\in\V_{(0,1)}$, so $u$ solves  \eqref{SDobs}. Now, if $v$ solves \eqref{SDobs} then
\begin{align*}
v_i \in (0,1) &\Rightarrow \beta_i = 0 \Rightarrow v_i = \frac{\left(e^{-\tau\Delta}u_n\right)_i -\lambda/2}{1-\lambda}
,\\
v_i = 0  &\Rightarrow \beta_i = \frac{1}{2}-\lambda^{-1}\left(e^{-\tau\Delta}u_n\right)_i\geq  \frac{1}{2},\\
v_i =1  &\Rightarrow \beta_i = -\frac{1}{2} + \lambda^{-1}\left(1-\left(e^{-\tau\Delta}u_n\right)_i \right)\leq - \frac{1}{2}.
\end{align*}
 Hence \eqref{SDobs} has as solution any $v\in \V$ obeying $v_i\in\{0,u_i,1\}$ for all $i\in V$. 

\end{nb}
\subsection{A Lyapunov functional for the semi-discrete flow}
In \cite[Proposition 4.6]{vGGOB} it was proved that the strictly concave functional \begin{equation*}
	J(u) := \ip{\mathbf{1}-u,e^{-\tau\Delta}u}
\end{equation*}
has first variation at $u$ \[ L_u(v) := \left\langle v,\mathbf{1}-2e^{-\tau\Delta}u\right \rangle_\V\] and hence is monotonically decreasing along MBO trajectories. Performing a calculation similar to that of \eqref{lamb1}, we find the functional in \eqref{ACobsMM} is equivalent to \[ F_{u_n}(u):= L_{u_n}(u) + (\lambda-1)\ip{u,\mathbf{1}-u}. \] 
We can therefore deduce a Lyapunov functional for the semi-discrete flow. Using a similar approach to Bertozzi and Luo's analysis of semi-implicit graph Allen--Cahn in \cite{LB2016} we furthermore prove results about long-time behaviour of the semi-discrete flow.
\begin{thm}\label{Lyapthm}
	When $0\leq\lambda\leq 1$ the functional \emph{(}on $\V_{[0,1]}$\emph{)}
	\begin{equation}\label{Lyap}
	H(u):= J(u) + (\lambda-1)\ip{u,\mathbf{1}-u} =\hspace*{-2pt}\footnotemark\, \lambda\ip{u,\mathbf{1}-u} +\left\langle u,\left(I-e^{-\tau\Delta}\right)u\right\rangle_\V
	\end{equation}\footnotetext{Since $e^{-\tau\Delta}$ is self-adjoint and $e^{-\tau\Delta}\mathbf{1}=\mathbf{1}$, $J(u) = \ip{\mathbf{1},u} - \ip{u,e^{-\tau\Delta} u}$ and so it follows that
$J(u) + (\lambda-1)\ip{u,\mathbf{1}-u} = \ip{\mathbf{1},u} - \ip{u,e^{-\tau\Delta} u} + \lambda\ip{u,\mathbf{1}-u} - \ip{\mathbf{1},u} + \ip{u,u}$.}
is non-negative, and furthermore $H$ is a Lyapunov functional for \eqref{SDobs}, i.e. $H(u_{n+1}) \leq H(u_n)$ with equality if and only if $u_{n+1}=u_n$ for $u_{n+1}$ defined by \eqref{SDobs}. In particular, \begin{equation}\label{Hstep}
		H(u_n)-H(u_{n+1}) \geq (1-\lambda)\left|\left|u_{n+1}-u_n\right|\right|^2_\V.
	\end{equation} 
\end{thm}
\begin{proof}
	Note that $I-e^{-\tau\Delta}$ has eigenvalues $1- e^{-\tau\lambda_k}\geq 0$, since $\lambda_k$ the eigenvalues of $\Delta$ are non-negative, and so $\left\langle u,\left(I-e^{-\tau\Delta}\right)u\right\rangle_\V
\geq 0$. Since $u\in\V_{[0,1]}$ it follows that $H(u)\geq 0$. 
	
	Next by the concavity of $J$ and linearity of $L_{u_n}$ we have:
	\begin{equation*}\begin{split}
	H(u_{n}) - H(u_{n+1}) &= J(u_{n}) - J(u_{n+1}) + (1-\lambda)\ip{u_{n+1},\mathbf{1}-u_{n+1}} - (1-\lambda)\ip{u_n,\mathbf{1}-u_n} \\
	&\geq L_{u_n}(u_{n} - u_{n+1}) + (1-\lambda)\ip{u_{n+1},\mathbf{1}-u_{n+1}} - (1-\lambda)\ip{u_n,\mathbf{1}-u_n} \:\left(*\right)\\
	&=F_{u_n}(u_{n}) - F_{u_n}(u_{n+1}) \geq 0 \text{ by \eqref{ACobsMM}}
	\end{split}
	\end{equation*}
	with equality if and only if $u_{n+1}=u_n$ as $J$ is strictly convex. Finally, we can continue
	\begin{equation*}\begin{split}
	\left(*\right) &= \ip{u_{n} - u_{n+1},\mathbf{1}-2e^{-\tau\Delta}u_n} + (1-\lambda)\ip{u_{n+1},\mathbf{1}-u_{n+1}} - (1-\lambda)\ip{u_n,\mathbf{1}-u_n} \\
	&= \ip{u_{n} - u_{n+1},\mathbf{1}-2e^{-\tau\Delta}u_n} + (1-\lambda)(\ip{u_{n+1}-u_n,\mathbf{1}} +\ip{u_n,u_n} -\ip{u_{n+1},u_{n+1}})\\
	&= \ip{u_{n} - u_{n+1},\lambda\mathbf{1}-2e^{-\tau\Delta}u_n + (1-\lambda)u_{n+1} + (1-\lambda)u_n}\\
	&= \ip{u_{n} - u_{n+1},2\lambda\beta_{n+1}+ (1-\lambda)(u_n-u_{n+1})} \text{ by (\ref{SDobs}b), recall }\beta_{n+1}\in\mathcal{B}(u_{n+1})\\
	&\geq (1-\lambda)\left|\left|u_{n+1}-u_n\right|\right|^2_\V
	\end{split}
	\end{equation*}
where the final line follows from \eqref{beta} as in the proof of Theorem \ref{obsMMthm}.
\end{proof}
\begin{cor}[Cf. $\text{\cite[Lemma 4]{LB2016}}$]\label{cor1}
	For $0\leq\lambda \leq 1$, we have that for the sequence $u_n$ given by \eqref{SDobs}  \[\sum_{n=0}^\infty \left|\left|u_{n+1}-u_n\right|\right|^2_\V < \infty \] and therefore in particular \[\lim_{n\rightarrow\infty} \left|\left|u_{n+1}-u_n\right|\right|_\V = 0. \]
\end{cor}
\begin{proof}
	If $\lambda=1$ the result follows directly from Theorem \ref{obsMMthm} since MBO trajectories are eventually constant \cite[Proposition 4.6]{vGGOB}. If $0\leq\lambda<1$ then by $H\geq 0$ and \eqref{Hstep} we have 
	\[(1-\lambda)\sum_{n=0}^N \left|\left|u_{n+1}-u_n\right|\right|^2_\V \leq H(u_0) -H(u_{N+1}) \leq H(u_0)\] so the result follows by taking $N\rightarrow\infty$.
\end{proof}
\begin{prop}
The Lyapunov functional has Hilbert space gradient \emph{(}for $u\in\V_{(0,1)}$\emph{)}
\be\label{Hgrad}
\nabla_\V H(u) = \lambda\mathbf{1} - 2e^{-\tau\Delta}u +2(1-\lambda)u.
\ee
%and therefore\emph{:} 
\begin{enumerate}[\em i.]\item For the sequence $u_n\in\V_{(0,1)}$ given by \eqref{SDobs}
\be\label{Hgrad2}
\nabla_\V H(u_n) = 2\lambda\beta_{n+1} +2(1-\lambda)(u_n-u_{n+1}). 
\ee
\item Let $\mathscr{E}$ denote the eigenspace of $\Delta$ with eigenvalue $-\tau^{-1}\log(1-\lambda)$ \emph{(}i.e. the eigenspace of $e^{-\tau\Delta}$ with eigenvalue $1-\lambda$\emph{)} or $\{\mathbf{0}\}$ if there is no such eigenvalue. Then, if $u\in\V_{(0,1)}$, it follows that $\nabla_\V H(u) =\mathbf{0}$ if and only if  $u\in \left(\frac{1}{2}\mathbf{1} + \mathscr{E}\right)\cap\V_{(0,1)}$. 
\end{enumerate}
\end{prop}
\begin{proof}
It is straightforward to check that \[ \ip{\nabla_\V H(u),v}:= \lim_{t\rightarrow0}\frac{ H(u+tv) - H(u)}{t}   = \ip{\mathbf{1}-2e^{-\tau\Delta}u,v}+(\lambda-1)\ip{\mathbf{1}-2u,v} \] and therefore \[\nabla_\V H(u) =\mathbf{1}-2e^{-\tau\Delta}u + (\lambda-1)(\mathbf{1}-2u)= \lambda\mathbf{1} - 2e^{-\tau\Delta}u +2(1-\lambda)u.\]
\begin{enumerate}[i.]\item By \eqref{SDobs}, $ \lambda \mathbf{1}- 2e^{-\tau\Delta}u_n =  2\lambda\beta_{n+1}-2(1-\lambda)u_{n+1}$, and substituting into \eqref{Hgrad} gives \eqref{Hgrad2}.
\item Let $A := 2e^{-\tau\Delta} +2(\lambda-1)I$, and note that $\mathscr{E}= \operatorname{ker} A$ by definition. Then by \eqref{Hgrad}, $u\in\V_{(0,1)}$ satisfies $\nabla_\V H(u) =\mathbf{0}$ if and only if $Au=\lambda\mathbf{1}$. Note that $\frac{1}{2}A\mathbf{1} =\mathbf{1} +\lambda\mathbf{1}-\mathbf{1}=\lambda\mathbf{1}$. Therefore $Au=\lambda\mathbf{1}$ if and only if $u\in \frac{1}{2}\mathbf{1} +\mathscr{E}$. %Therefore $u\in\V_{(0,1)}$ satisfies $\nabla_\V H(u) =\mathbf{0}$ if and only if $u\in \left(\frac{1}{2}\mathbf{1} + \mathscr{E}\right)\cap\V_{(0,1)}$. 
\end{enumerate}
\end{proof}
\begin{nb}
Considering the quadratic terms one can observe that \[ H\left(\frac{1}{2}\mathbf{1}+\eta\right) = H\left(\frac{1}{2}\mathbf{1}\right) - \left(\ip{\eta,e^{-\tau\Delta}\eta}-(1-\lambda)\ip{\eta,\eta}\right) \] so, for $0\leq\lambda\leq 1$, $u= \frac{1}{2}\mathbf{1}$ is a global maximiser of $H$ if and only if \[P := e^{-\tau\Delta} -(1-\lambda)I\] is positive semi-definite. For $\lambda_k$ the eigenvalues of $\Delta$, we desire $P$ have eigenvalues{:} 
\begin{align*}& e^{-\tau\lambda_k} -(1-\lambda)\geq 0 &\text{i.e.} && \tau \varepsilon^{-1} \geq 1 - e^{-\tau\lambda_k}  .\end{align*} Therefore we have $\lambda\leq 1$ and $u= \frac{1}{2}\mathbf{1}$ as a global maximiser of $H$ if and only if 
\[\varepsilon\in \left[\tau,\frac{\tau}{1-e^{-\tau||\Delta||}}\right].\]
Furthermore, for $\xi\in\mathscr{E}$ we have $e^{-\tau\Delta}\xi = (1-\lambda)\xi$, and so \[ H\left(\frac{1}{2}\mathbf{1}+\xi\right) = H\left(\frac{1}{2}\mathbf{1}\right) - \left(\ip{\xi,e^{-\tau\Delta}\xi}-(1-\lambda)\ip{\xi,\xi}\right) =  H\left(\frac{1}{2}\mathbf{1}\right). \] Therefore $\left(\frac{1}{2}\mathbf{1} + \mathscr{E}\right)\cap\V_{(0,1)}$ are all global maxima in this case.
\end{nb}

We summarise the payoff from the above discussion. As $H(u_n)$ is monotonically decreasing and bounded below, we have $H(u_n)\downarrow H_\infty$ for some $H_\infty\geq 0$. Furthermore, since the $u_n\in\V_{[0,1]}$, which is compact, there exist subsequences $u_{n_k}\rightarrow u^*\in\V_{[0,1]}$ with $H(u^*)=H_\infty$. Unfortunately, as in \cite{LB2016} for the AC flow with the standard quartic potential, these facts are insufficient for convergence of the $u_n$ as $n\rightarrow\infty$. However, by the same argument as \cite[Lemma 5]{LB2016}, if the $u_n$ have finitely many accumulation points, then the full sequence $u_n$ converges. Notably, if $u^*\in\V_{(0,1)}$ is an accumulation point of $u_n$ then by Corollary \ref{cor1} and \eqref{Hgrad2} we have $\nabla_\V H(u^*) = \mathbf{0}$, and so $H(u^*) =H(\frac{1}{2}\mathbf{1})$. So if $H(u_0) < H(\frac{1}{2}\mathbf{1})$, then no accumulation points of $u_n$ lie in $\V_{(0,1)}$.

\begin{nb}
We can suggestively generalise the Lyapunov functional to all of $\V$ by defining \[ 
%\begin{split}
	H(u) := \left\langle u,u-e^{-\tau\Delta}u\right\rangle_\V + 2\lambda\ip{W\circ u,\mathbf{1}}\geq 0			  
%\end{split}
\]
which we can rewrite as: \[ H(u)= \tau\ip{u,\Delta u} + 2\lambda\ip{W\circ u,\mathbf{1}} - \tau^2\ip{u,Q_\tau u}\]
where $Q_\tau$ satisfies $e^{-\tau\Delta} = I -\tau\Delta + \tau^2 Q_\tau$ and therefore \begin{equation}\label{SDLyap}
 	\frac{1}{2\tau}H(u) = \GL(u) - \frac{1}{2}\tau\ip{u,Q_\tau u}.
 \end{equation}
Note that $Q_\tau$ has eigenvalues $\tau^{-2}(e^{-\tau\lambda_k}-1+\tau\lambda_k)\leq\frac{1}{2}\lambda_k^2$, for $\lambda_k$ the eigenvalues of $\Delta$, and so $\ip{u,Q_\tau u}\leq \frac{1}{2}||\Delta||^2||u||_\V^2$  is bounded in $\V_{[0,1]}$ uniformly in $\tau$. Hence for $u\in\V_{[0,1]}$
\[
 	\frac{1}{2\tau}H(u) = \GL(u) +\bigO(\tau).
\]
\end{nb}

\subsection{The semi-discrete scheme and a time-splitting of the AC flow}
Theorem \ref{obsMMthm} shows that the MBO scheme is the AC flow approximated via our semi-discrete scheme \eqref{SDobs}. The semi-discrete scheme can also be related to the following two-step time-splitting of the AC flow.\footnote{The link between the MBO scheme and this time-splitting was also noted in the continuum case in \cite{ET}.} We fix $\tau>0$ and take $\tilde u_0\in\V_{[0,1]}$, then iteratively apply the steps:

\begin{enumerate}
\item Take $\tilde  u_n$ from the previous iteration. 
\item\textbf{(Diffusion step)} Define $v:=e^{-t\Delta}\tilde u_n$ the heat equation solution with $v(0) =\tilde u_n$ and define $v_n := v(\tau)$.

\item\textbf{(Reaction step)} Define $U_{n}\in  H^1((0,\tau);\V)\cap C^0([0,\tau];\V)\cap\V_{[0,1],t\in [0,\tau]}$ obeying
\be
\label{thresh}
\frac{d(U_n)_i}{dt} = \varepsilon^{-1}\left((U_n(t))_i - \frac{1}{2}\right) + \varepsilon^{-1}\beta_i(t), \: \: \beta\in\mathcal{B}(U_n), \: \: U_{n}(0) = v_n =e^{-\tau\Delta}\tilde u_n.
\ee
%where $\beta\in\mathcal{B}(U_n)$ as defined in \eqref{beta}. 
\item Finally, define $\tilde u_{n+1}:=U_n(\tau)$ and define $U(t):=U_n(t-n\tau)$ for $t\in (n\tau,(n+1)\tau]$. 
\end{enumerate}
The relation of this time-splitting to the semi-discrete scheme is that if $ u_n = \tilde u_n$ we can recognise from \eqref{SDobs0} the semi-discrete update $u_{n+1}\approx \tilde u_{n+1}$ as the implicit Euler approximation of \eqref{thresh}:
\[
\label{IE}
\frac{(u_{n+1})_i-(v_n)_i}{\tau} \in - \frac{1}{\varepsilon} \partial W((u_{n+1})_i).
\]
That is, we get the semi-discrete update by dissecting the flow in \eqref{ACobs2} into a diffusion for time $\tau$ followed by a gradient flow of $W$, again for time $\tau$. Then, we use the exact solution for the former and approximate the latter by an implicit Euler scheme. 

We next prove that minimisers of $W$ are stationary for \eqref{thresh}, a gradient flow of $W$.
\begin{prop}
Let $x:[0,T]\rightarrow[0,1]$ with $x\in H^1((0,T);\mathbb{R})\cap C^0([0,T];\mathbb{R})$ solve \[\varepsilon\frac{dx}{dt} =x(t) -\frac{1}{2} + \beta(t) \] with $x(0) \in \{0,1\}$ and $\beta(t) \in -\partial I_{[0,1]}(x(t))$. Then $x(t) = x(0)$ for all $t\in[0,T]$.
\end{prop}
\begin{proof} By symmetry, it suffices to consider $x(0) = 0$. Let $T_0$ equal the first time when $x(t) = 1/2$, or let $T_0 = T$ if there is no such time. Then for $t\in[0,T_0]$ we have $0\leq x(t)\leq 1/2$, and note that for any $t\in[0,T]$, $x(t)\beta(t)\leq 0$. Hence, for all $t\in [0,T_0]$, \[\frac{1}{2}\varepsilon\frac{d}{dt}\left(x(t)^2\right) = \varepsilon x(t) \frac{dx}{dt} = x(t)\left(x(t) -\frac{1}{2}\right) + \beta(t) x(t)  \leq 0 \] and thus $x(t)^2\leq x(0)^2 = 0$. Therefore for all $t\in [0,T_0]$, $x(t) = 0$. Since $x(T_0)\neq 1/2$, we must have $T_0=T$, completing the proof.
\end{proof} 

We can therefore solve \eqref{thresh}. Let $T_i$ denote the first time $(U_n(t))_i\in\{0,1\}$. Then for $t<T_i$, $\beta_i(t) = 0$, so by separation of variables we have
\[(U_n(t))_i -\frac{1}{2} =\left((v_n)_i -\frac{1}{2}\right)e^{t/\varepsilon}\] and for $t\geq T_i$, $(U_n(t))_i = (U_n(T_i))_i$ by the above proposition. Thus
\begin{equation*}
	\label{threshsoln}
	(\tilde u_{n+1})_i = \begin{cases}
\frac{1}{2}+e^{\lambda}\left((v_n)_i -\frac{1}{2}\right), &\text{if }e^{\lambda}\left|(v_n)_i -\frac{1}{2}\right|<\frac{1}{2}\\
\Theta\left((v_n)_i -\frac{1}{2}\right) , &\text{otherwise}
 \end{cases}
\end{equation*}
where $\Theta$ is the Heaviside step function. We compare to the semi-discrete update $u_{n+1}$. If $u_n=\tilde u_n$ then by \eqref{ACobssoln} the semi-discrete update obeys: 
\[	(u_{n+1})_i = \begin{cases}
 \frac{1}{2}+ \frac{1}{1-\lambda}\left((v_n)_i -\frac{1}{2}\right), &\text{ if } \lambda < 1 \text{ and } \frac{1}{1-\lambda}\left|(v_n)_i -\frac{1}{2}\right|<\frac{1}{2},\\
\Theta\left((v_n)_i -\frac{1}{2}\right), &\text{ otherwise}.
 \end{cases}\\
\]
For $0\leq\lambda\leq 1$, $e^{\lambda}\leq(1-\lambda)^{-1}$. Thus $|(\tilde u_{n+1})_i-1/2|\leq|(u_{n+1})_i -1/2|$, so after one step the semi-discrete update is always at least as close to the wells of $W$ as the time-splitting update. 

We compare the time-splitting flow to the AC flow. Note the simplest case, when $u$ remains in $\V_{(0,1)}$ so $\beta$ is constantly $ \mathbf{0}$. Then substituting into \eqref{ACobssoln2}:
\begin{equation}
\label{ACnoobssoln}
u(t) =\frac{1}{2}\mathbf{1} +  e^{t/\varepsilon}\left(e^{-t\Delta}u(0)-\frac{1}{2}\mathbf{1}\right).
\end{equation}
 Note that in this case we also have for $t\in [0,\tau]$ and with $\tilde u_0=u(0)$,
\[U_0(t)= \frac{1}{2}\mathbf{1} +e^{t/\varepsilon}\left(e^{-\tau\Delta}u(0) -\frac{1}{2}\mathbf{1}\right)\]
 so $\tilde u_1 = u(\tau)$, and so by induction $\tilde u_n=u(n\tau)$, i.e. the time-splitting scheme agrees with the AC flow after each time-step of $\tau$. Note also that by {\cite[Lemma 2.6(d)]{vGGOB}}, for $t\in (n\tau,(n+1)\tau]$, we have $||e^{-(t-n\tau)\Delta}u(n\tau)-\frac{1}{2}\mathbf{1}||_\infty \geq ||e^{-\tau\Delta}u(n\tau)-\frac{1}{2}\mathbf{1}||_\infty $, and so
\begin{equation}\label{TSineq}
	\left|\left|u(t)-\frac{1}{2}\mathbf{1}\right|\right|_\infty \geq \:\;\left|\left|U(t)-\frac{1}{2}\mathbf{1}\right|\right|_\infty
\end{equation} with strict inequality when $u(n\tau)$ is not a constant vector and $t\in(n\tau,(n+1)\tau)$, so the AC flow is always at least as close to the wells of $W$ as the time-splitting flow. 
However in general $u(t)$ will eventually have a vertex take value $0$ or $1$, and indeed if $\ip{u(0),\mathbf{1}}\neq \frac{1}{2}\ip{\mathbf{1},\mathbf{1}}$ then \eqref{ACnoobssoln} yields an upper bound on how long until this occurs.\footnote{By \cite[Lemma 2.6(c)]{vGGOB}, for all $i\in V$ $(e^{-t\Delta}u(0))_i -\frac{1}{2}\rightarrow \ip{u(0)-\frac{1}{2}\mathbf{1},\mathbf{1}}\ip{\mathbf{1},\mathbf{1}}^{-1}=:\alpha$ as $t\rightarrow\infty$. From that lemma, we have an explicit $t_0$ such that for all $t> t_0$, $|(e^{-t\Delta}u(0))_i -\frac{1}{2}| >\frac{1}{2}|\alpha  |$. Then by \eqref{ACnoobssoln}, for $t>\max\{t_0, -\varepsilon\log| \alpha  |\}$ we have $u_i(t)\notin(0,1)$. 
} 
Furthermore, and importantly, by \eqref{TSineq} it will do so before this happens for $U(t)$.
\section{Convergence of the semi-discrete scheme to the AC flow}\label{SDconvsec}

We now consider the behaviour of the semi-discrete solution as $\tau,\lambda\downarrow 0$ ($\varepsilon$ fixed). Drawing inspiration from the form of the AC solution in \eqref{ACobssoln2}, we show that the semi-discrete solutions converge pointwise (in a sense we make explicit) to the AC solution. 

In summary, the proof runs by finding the $n^\text{th}$ term for the semi-discrete scheme, and noticing that this asymptotically resembles a Riemann sum for the integral form of the AC solution in \eqref{ACobssoln2}. We make the convergence of this sum precise via the Banach--Alaoglu theorem, and then use the Banach--Saks theorem to extract a valid $\beta$ term from the sequence of semi-discrete $\beta_n$. 
\subsection{Asymptotics of the semi-discrete iterates}
Recall that the semi-discrete scheme \eqref{SDobs} is defined as \[(1-\lambda)u_{n+1} =e^{-\tau\Delta}u_n-\frac{\lambda}{2}\mathbf{1} +\lambda\beta_{n+1}
\] and note that $\beta_{n+1}\in \V_{[-1/2,1/2]}$ by \eqref{betasoln}. 
Iterating, we get the $n^\text{th}$ term.
\begin{prop}For $\lambda :=\tau/\varepsilon \in [0,1)$, the semi-discrete solution has $n^\text{th}$ iterate \be u_n = \frac{1}{2}\mathbf{1}+(1-\lambda)^{-n}e^{-n\tau\Delta}\left(u_0-\frac{1}{2}\mathbf{1}\right) +\frac{\lambda}{1-\lambda}\sum_{k=1}^n (1-\lambda)^{-(n-k)}e^{-(n-k)\tau\Delta}\beta_k . \ee
Understanding $\bigO$ to refer to the limit of $\tau\downarrow 0$ and $n\rightarrow \infty$ with $n\tau-t\in [0,\tau)$ for some fixed $t\geq 0$ and for fixed $\varepsilon>0$, we therefore have
\begin{equation}\label{SDiterate}
	u_n = \frac{1}{2}\mathbf{1}+e^{n\lambda} e^{-n\tau\Delta}\left(u_0-\frac{1}{2}\mathbf{1}\right) +\lambda\sum_{k=1}^n e^{(n-k)\lambda}e^{-(n-k)\tau\Delta}\beta_k +\bigO(\tau).
\end{equation}

\end{prop}
\begin{proof}
	We proceed by induction. The $n = 0$ base case is trivial. Then inducting: 
\vspace{-0.1em}\begin{equation*}\begin{split}&u_{n+1} = (1-\lambda)^{-1}e^{-\tau\Delta}u_n -\frac{1}{2}\frac{\lambda}{1-\lambda}\mathbf{1} +\frac{\lambda}{1-\lambda}\beta_{n+1}\\
								   &=(1-\lambda)^{-1}e^{-\tau\Delta}\left[\frac{1}{2}\mathbf{1}+(1-\lambda)^{-n}e^{-n\tau\Delta}\left(u_0-\frac{1}{2}\mathbf{1}\right) +\frac{\lambda}{1-\lambda}\sum_{k=1}^n (1-\lambda)^{-(n-k)}e^{-(n-k)\tau\Delta}\beta_k\right]\\
								   &\hspace{1em } -\frac{1}{2}\frac{\lambda}{1-\lambda}\mathbf{1} +\frac{\lambda}{1-\lambda}\beta_{n+1}\\
								   &=\frac{1}{2}\left(\frac{1}{1-\lambda}-\frac{\lambda}{1-\lambda}\right)\mathbf{1}+(1-\lambda)^{-(n+1)}e^{-(n+1)\tau\Delta}\left(u_0-\frac{1}{2}\mathbf{1}\right)\\
								   &\hspace{1em } + \frac{\lambda}{1-\lambda}\sum_{k=1}^n (1-\lambda)^{-(n-k+1)}e^{-(n-k+1)\tau\Delta}\beta_k +\frac{\lambda}{1-\lambda}\beta_{n+1}\\
								   &=\frac{1}{2}\mathbf{1}+(1-\lambda)^{-(n+1)}e^{-(n+1)\tau\Delta}\left(u_0-\frac{1}{2}\mathbf{1}\right)+\frac{\lambda}{1-\lambda}\sum_{k=1}^{n+1} (1-\lambda)^{-(n-k+1)}e^{-(n-k+1)\tau\Delta}\beta_k 
									 \end{split}\end{equation*}
	completing the induction. 

Next, consider (with $n\tau=t+\bigO(\tau)$ and $n\lambda=t/\varepsilon+\bigO(\tau)$ for fixed $t$) the difference:
\begin{equation*}\begin{split}&\left|\left|u_n - \frac{1}{2}\mathbf{1} - e^{n\lambda} e^{-n\tau\Delta}\left(u_0-\frac{1}{2}\mathbf{1}\right)-\lambda\sum_{k=1}^n e^{(n-k)\lambda}e^{-(n-k)\tau\Delta}\beta_k\right|\right|_\V \\
	&= \left|\left|\left((1-\lambda)^{-n}-e^{n\lambda}\right)\left(e^{-n\tau\Delta}u_0-\frac{1}{2}\mathbf{1}\right)+\lambda \sum_{k=1}^n \left(\frac{(1-\lambda)^{-(n-k)}}{1-\lambda}- e^{(n-k)\lambda}\right)e^{-(n-k)\tau\Delta}\beta_k\right|\right|_\V\\
	%&\hspace{2em } + \frac{\lambda^2}{1-\lambda} \sum_{k=1}^n (1-\lambda)^{-(n-k)}e^{-(n-k)\tau\Delta}\beta_k\bigg|\bigg|_\V\\
	&\leq\left((1-\lambda)^{-n}-e^{n\lambda}\right)\left|\left|e^{-n\tau\Delta}u_0-\frac{1}{2}\mathbf{1} \right|\right|_\V +\lambda \sum_{k=1}^n \left(\frac{(1-\lambda)^{-(n-k)}}{1-\lambda}- e^{(n-k)\lambda}\right)\left|\left|e^{-(n-k)\tau\Delta}\beta_k\right|\right|_\V\\ 
	&\leq \left((1-\lambda)^{-n}-e^{n\lambda}\right)\left(\left|\left|e^{-t\Delta}u_0-\frac{1}{2}\mathbf{1} \right|\right|_\V +\bigO(\tau)\right)+  C\frac{\lambda}{1-\lambda} \sum_{k=1}^n (1-\lambda)^{-(n-k)}-\lambda C \sum_{k=1}^n e^{(n-k)\lambda}.\end{split}\end{equation*}
The first inequality follows from the triangle inequality since, for $r:=n-k\geq 0$, $(1-\lambda)^{-(r+1)} - e^{r\lambda}\geq 0$ as $ e^{-\lambda r/(r+1)}\geq 1-\lambda r/(r+1) \geq 1-\lambda $, and the second follows by taking \[C:=\sup_{s\in[0,t+\tau]}\sup_{\beta\in\V_{[-1/2,1/2]}}\left|\left|e^{-s\Delta}\beta \right|\right|_\V\leq \sup_{s\in[0,t+\tau]}\left|\left|e^{-s\Delta} \right|\right|\cdot \frac{1}{ 2}\left|\left|\mathbf{1} \right|\right|_\V=\frac{1}{ 2}\left|\left|\mathbf{1} \right|\right|_\V.\]
We continue, 
\begin{equation*}
	RHS= \left(D+\bigO(\tau)\right)\left(\left(1-\lambda\right)^{-n}-e^{t/\varepsilon} \right)+ C\frac{\lambda}{1-\lambda} \frac{(1-\lambda)^{-n}-1}{(1-\lambda)^{-1}-1}- \lambda C \frac{e^{n\lambda}-1}{e^{\lambda}-1}
\end{equation*} 
where $D:=\left|\left|e^{-t\Delta}u_0-\frac{1}{2}\mathbf{1} \right|\right|_\V$. Hence, noting that $(1-\lambda)^{-n} = (1-(t/\varepsilon+\bigO(\tau))/n)^{-n}=e^{t/\varepsilon+\bigO(\tau)} +\bigO(1/n) = e^{t/\varepsilon} +\bigO(\tau) $, we finally have
\begin{equation*}
	RHS=\bigO(\tau) + C \left(e^{t/\varepsilon}  - 1 +\bigO(\tau)-\lambda \frac{e^{t/\varepsilon} -1}{e^{\lambda}-1}\right)
	=C\left(e^{t/\varepsilon}  - 1\right)\frac{e^\lambda-1-\lambda}{e^\lambda-1} +\bigO(\tau) =\bigO(\tau) \end{equation*} as desired.
\end{proof}
\subsection{Proof of convergence to the AC solution}
We consider the limit of \eqref{SDiterate} as $\tau\downarrow 0$, $n\rightarrow \infty$ with $n\tau\rightarrow t$ for some fixed $t$ and $\tau<\varepsilon$. As a prelude to this, for reasons that will soon become clear, we define the piecewise constant function $z_\tau:[0,\infty)\rightarrow\V$
\[z_\tau(s): =\begin{cases} e^{-\tau/\varepsilon} e^{\tau\Delta}\beta^{[\tau]}_1, & 0\leq s \leq \tau\\
									 e^{-k\tau/\varepsilon} e^{k\tau\Delta} \beta^{[\tau]}_k,	& (k-1)\tau<s\leq k\tau \text{ for } k\in\mathbb{N}						
\end{cases}\]
and the function \[\gamma_\tau(s):= e^{s/\varepsilon}e^{-s\Delta} z_\tau(s)=\begin{cases} e^{-(\tau-s)/\varepsilon} e^{(\tau-s)\Delta}\beta^{[\tau]}_1, & 0\leq s \leq \tau\\
									 e^{-(k\tau-s)/\varepsilon} e^{(k\tau-s)\Delta} \beta^{[\tau]}_k,	& (k-1)\tau<s\leq k\tau			\text{ for } k\in\mathbb{N}				
\end{cases}\]
(note that for bookkeeping we have introduced the superscript $[\tau]$ to keep track of the time-step governing a particular sequence of $u_n$ and $\beta_n$). We note an important convergence.
\begin{prop}
For any sequence $\tau^{(0)}_n\downarrow 0$ with $\tau^{(0)}_n<\varepsilon$ for all $n$,  there exists a function $z:[0,\infty) \rightarrow \V$ and a subsequence $\tau_n$ such that $z_{\tau_n}$ converges weakly to $z$ in $L^2_{loc}$.
\end{prop}
\begin{proof} For $N\in \mathbb{N}$, consider  $z_\tau|_{[0,N]}$. As the $\beta_k^{[\tau]}\in\V_{[-1/2,1/2]}$ for all $k$ and $\tau$, we have for all $s\in[0,N]$ and $0\leq\tau<\varepsilon$
\begin{equation*}\left|\left|z_\tau(s)\right|\right|_\V \leq \sup_{s'\in[0,N+\varepsilon]}\left|\left|e^{-s'(\frac{1}{\varepsilon}I-\Delta)}\right|\right| \cdot \frac{1}{2} \left|\left|\mathbf{1}\right|\right|_\V \leq \max\left\{1, e^{(N+\varepsilon)(||\Delta||-\varepsilon^{-1})}\right\} \cdot \frac{1}{2} \left|\left|\mathbf{1}\right|\right|_\V.\footnotemark
\end{equation*}
\footnotetext{Here we have used that for $s\leq N$ the corresponding $k\tau$ in the exponent of $z_\tau(s)$ is less than $N+\tau$, and that $||e^{-s'(\frac{1}{\varepsilon}I-\Delta)}||=e^{s'(||\Delta||-\varepsilon^{-1})}$ is maximised at the endpoints of $[0,N+\varepsilon]$. }
Thus the $z_\tau|_{[0,N]}$ are uniformly bounded for $0\leq\tau<\varepsilon$, and hence lie in a closed ball in $L^2([0,N];\V)$. By the Banach--Alaoglu theorem on a Hilbert space, this ball is weak-compact. We proceed by a ``local-to-global" diagonal argument. For $N = 1$: by compactness we have a subsequence $\tau^{(1)}$ of $\tau^{(0)}$ such that $z_{\tau^{(1)}_n}$ converges weakly to $z$ on $[0,1]$. Iterating, by compactness we have a subsequence $\tau^{(N+1)}$ of $\tau^{(N)}$ such that $z_{\tau^{(N+1)}_n}$ converges weakly to $z$ on $[0,N+1]$. Finally, define $\tau_n : = \tau_n^{(n)}$. Then for all bounded $T\subseteq [0,\infty)$, we have $T\subseteq [0,M]$ for some $M\in \mathbb{N}$ and hence $z_{\tau_n}|_T$ is eventually a subsequence of $z_{\tau^{(M)}_n}|_T$, so converges weakly to $z|_T$.
\end{proof}
\begin{cor} \label{zcor} From $z_{\tau_n} \rightharpoonup z$ in $L^2_{loc}$ we infer\emph{:}
\begin{enumerate}[ A.]
\item $\gamma_{\tau_n}\rightharpoonup \gamma $ \emph{(}in $L^2_{loc}$\emph{)} where $\gamma(s) := e^{s/\varepsilon}e^{-s\Delta} z$.
\item For all $t\geq 0$, \[\int_0^t z_{\tau_n}(s)\; ds \rightarrow \int_0^t z(s)\; ds.\]
\item Replacing $\tau_n$ by a subsequence, we have strong convergence of the Cesàro sums, i.e. for all bounded $T\subseteq[0,\infty)$, as $N\rightarrow \infty$ \begin{align*}&\frac{1}{N}\sum_{n=1}^N z_{\tau_n} \rightarrow z, &&\frac{1}{N}\sum_{n=1}^N \gamma_{\tau_n} \rightarrow \gamma, && \text { in } L^2(T;\V).\end{align*} 
\end{enumerate}
\end{cor}
\begin{proof}
(A) follows since $f\mapsto e^{s/\varepsilon}e^{-s\Delta} f$ is a continuous self-adjoint map on $L^2(T;\V)$ for $T$ bounded. Hence for all $f\in L^2(T;\V)$, 
\[(\gamma_{\tau_n},f)_{s\in T} = (z_{\tau_n},e^{s/\varepsilon}e^{-s\Delta}f)_{s\in T}\rightarrow (z,e^{s/\varepsilon}e^{-s\Delta}f)_{s\in T} = (\gamma,f)_{s\in T}.\]
(B) is a direct consequence of weak convergence. (C) follows by the Banach--Saks theorem \cite{BS} and a ``local-to-global" diagonal argument as in the above proof.
\end{proof}
%\begin{nb} Since $L^2(T;\V)$ is separable, the theorem of Banach--Alaoglu can be stated constructively, as can the theorem of Banach--Saks. Hence these proofs can be made constructive to yield an explicit $\tau_n$. We omit the details of this construction. 
%\end{nb}
We now return to the question of convergence of the semi-discrete iterates. Taking $\tau$ to zero along the sequence $\tau_n$, we define for all $t\geq 0$ the continuous-time function: 
\be \label{uhat}\hat u(t) := \lim_{n\rightarrow\infty, m =\ceil{t/\tau_n}} u^{[\tau_n]}_m. \ee
Therefore by \eqref{SDiterate}
\begin{equation*}
%\hspace{-3em}
\begin{split}
\hat u(t) = \frac{1}{2}\mathbf{1} + \lim_{n\rightarrow \infty} e^{m\tau_n/\varepsilon} e^{-m\tau_n\Delta}\left(u_0 -\frac{1}{2}\mathbf{1}\right) +\frac{1}{\varepsilon}e^{m\tau_n/\varepsilon} e^{-m\tau_n\Delta} \tau_n\sum_{k=1}^m e^{-k\tau_n/\varepsilon} e^{k\tau_n\Delta}\beta^{[\tau_n]}_k\\
\intertext{and by rewriting the sum term via the definition of $z_{\tau_n}$:}
\hat u(t)= \frac{1}{2}\mathbf{1} + \lim_{n\rightarrow \infty} e^{m\tau_n/\varepsilon} e^{-m\tau_n\Delta}\left(u_0 -\frac{1}{2}\mathbf{1}\right) +\frac{1}{\varepsilon}e^{m\tau_n/\varepsilon} e^{-m\tau_n\Delta}\int_0^{m\tau_n} z_{\tau_n}(s)\; ds.
\end{split}\end{equation*}
Then to prove global convergence we must show the following desiderata:
\begin{enumerate}[ i.]
\item $\hat u(t)$ exists for all $t\geq 0$,
\item $\hat u(t)\in\V_{[0,1]}$ for all $t\geq 0$,
\item $\hat u$ is continuous, $H^1_{loc}$, and is for a.e. $t\geq 0$ a solution to the AC ODE \eqref{ACobs2}. 
\end{enumerate}

Towards (i), let $A:= \varepsilon^{-1}I-\Delta$ and $e_n :=  m\tau_n - t \in [0,\tau_n)$. Then 
\[e^{m\tau_n/\varepsilon} e^{-m\tau_n\Delta} = e^{(t+e_n)A} = e^{tA}(I + \bigO(e_n)) = e^{tA} + \bigO(\tau_n)\]
and so 
\[\hat u(t) =  \frac{1}{2}\mathbf{1} + \lim_{n\rightarrow\infty} e^{tA}\left(u_0 -\frac{1}{2}\mathbf{1}\right)+\frac{1}{\varepsilon}\left(e^{tA} + \bigO(\tau_n)\right)\left(\int_0^{t} z_{\tau_n}(s)\; ds+\int_t^{t+e_n}  z_{\tau_n}(s) \; ds\right).  \]
Hence since the $z_{\tau_n}$ are uniformly bounded on $[0,t+\max_n e_n]$ and by Corollary \ref{zcor}(B)
\be \label{uhatsoln} \begin{split}\hat u(t) &=  \frac{1}{2}\mathbf{1} + e^{t/\varepsilon} e^{-t\Delta}\left(u_0 -\frac{1}{2}\mathbf{1}\right) + \frac{1}{\varepsilon}e^{t/\varepsilon} e^{-t\Delta}\int_0^{t} z(s)\; ds \\
&= \frac{1}{2}\mathbf{1} + e^{t/\varepsilon} e^{-t\Delta}\left(u_0 -\frac{1}{2}\mathbf{1}\right) + \frac{1}{\varepsilon}\int_0^{t} e^{(t-s)/\varepsilon} e^{-(t-s)\Delta} \gamma(s)\; ds.\end{split}\ee

To show (ii), we note simply that $\hat u(t)$ is a limit of semi-discrete iterates, each of which we know lies in $\V_{[0,1]}$. 

%To show (iii), we note that  $\int_0^t z(s)\; ds$ is continuous since $z$ is a weak limit of locally bounded functions, so is locally bounded, and hence $\hat u$ is continuous by \eqref{uhatsoln}. Next, by (ii) $\hat u$ is bounded so is locally $L^2$. It is easy to check that $\hat u$ has weak derivative \[\frac{d\hat u}{dt} = \left(\frac{1}{\varepsilon}I -\Delta\right)\left(u_0 -\frac{1}{2}\mathbf{1}\right) + \frac{1}{\varepsilon}e^{t/\varepsilon} e^{-t\Delta}\left(z(t) + \left(\frac{1}{\varepsilon}I -\Delta\right) \int_0^t z(s)\; ds\right)\] 
%which is locally $L^2$ since $z$ is a weak limit of locally $L^2$ functions (so is locally $L^2$) and $\int_0^t z(s)\; ds$ is a pointwise limit of locally bounded functions, so is locally bounded.

Finally to show (iii), we recall the conditions from Theorem \ref{ACexplicit}. By \eqref{uhatsoln} we have that $\hat u$ has the desired integral form, and since $\gamma$ is a weak limit of locally bounded and locally integrable functions we have that $\gamma$ is locally bounded a.e. and is locally integrable.
% we can rearrange the AC ODE into \[\frac{d}{dt}\left(e^{-t/\varepsilon}e^{t\Delta}\left(u(t)-\frac{1}{2}\mathbf{1}\right)\right)=\varepsilon^{-1}e^{-t/\varepsilon}e^{t\Delta}\beta(t).\] 
%Inspection of \eqref{uhatsoln} shows that to prove that $\hat u$ solves this ODE a.e.

So  it suffices to check that $\gamma(t) \in \mathcal{B}(\hat u(t))$ for a.e. $t\geq 0$. By Corollary \ref{zcor}(C) we have that on each bounded $T\subseteq[0,\infty)$, $\gamma$ is the $L^2(T;\V)$ limit as $N\rightarrow \infty$ of \[S_N:= \frac{1}{N} \sum_{n=1}^N \gamma_{\tau_n}.\] Since $L^2$ convergence implies a.e. pointwise convergence along a subsequence, by a ``local-to-global" diagonal argument we can extract a sequence $N_k \rightarrow \infty$ such that for a.e. $t\geq 0$ 
\[\gamma(t) = \lim_{k\rightarrow\infty}  \frac{1}{N_k} \sum_{n=1}^{N_k} \gamma_{\tau_n}(t).\]
Now recall $m := \ceil{t/\tau_n}$ and $e_n := m\tau_n -t \in [0,\tau_n)$. Then 
\[\gamma_{\tau_n}(t) = e^{-e_nA}\beta_m^{[\tau_n]} = \beta_m^{[\tau_n]} + \bigO(e_n) = \beta_m^{[\tau_n]} + \bigO(\tau_n)  \]
as the $\beta$ are uniformly bounded. Therefore for a.e. $t\geq 0$,
\[\gamma(t) =\lim_{k\rightarrow\infty}  \frac{1}{N_k} \sum_{n=1}^{N_k} \beta_m^{[\tau_n]}.\] 
Recall that $u^{[\tau_n]}_m\rightarrow\hat u(t)$ and $\beta_m^{[\tau_n]}\in \mathcal{B}( u^{[\tau_n]}_m)$. Suppose first that $\hat u_i(t) \in (0,1)$.  Then we have some $M$ such that for all $n>M$, $(u^{[\tau_n]}_m)_i\in (0,1)$ and so $(\beta^{[\tau_n]}_m)_i=0$. Hence 
\[\gamma_i(t) = \lim_{k\rightarrow\infty}  \frac{1}{N_k}\left( \sum_{n=1}^{M} (\beta_m^{[\tau_n]})_i +  \sum_{n=M+1}^{N_k} 0\right) = 0\] 
as desired. Next suppose $\hat u_i(t) = 0$. Then we have some $M$ such that for all $n>M$, $(u^{[\tau_n]}_n)_i\in [0,1)$ and so $ (\beta^{[\tau_n]}_n)_i \geq 0$. Hence 
\[\gamma_i(t) \geq \lim_{k\rightarrow\infty}  \frac{1}{N_k}\left( \sum_{n=1}^{M} (\beta_m^{[\tau_n]})_i + \sum_{n=M+1}^{N_k} 0 \right) =0\] 
as desired. Likewise for $\hat u_i(t) = 1$, $\gamma_i(t) \leq 0$. Hence we have $\gamma(t)\in\mathcal{B}(\hat u (t))$. 

In summary, we have the covergence result.
\begin{thm}\label{SDlimit}
For any given $u_0\in\V_{[0,1]}$, $\varepsilon>0$ and $\tau_n\downarrow 0$, there exists a subsequence $\tau'_n$ of $\tau_n$ with $\tau'_n<\varepsilon$ for all $n$, such that along this subsequence the semi-discrete iterates $(u^{[\tau'_n]}_m,\beta^{[\tau'_n]}_m)$ given by \eqref{SDobs} with initial state $u_0$ converge to the AC solution with initial condition $u_0$. That is, for each $t\geq 0$, as $n\rightarrow\infty$ and $m=\ceil{t/\tau'_n}$, $u^{[\tau'_n]}_m\rightarrow \hat u(t)$, and there is a sequence $N_k\rightarrow\infty$ \emph{(}independent of $t$\emph{)} such that for almost every $t\geq 0$, $\frac{1}{N_k}\sum_{n=1}^{N_k}\beta^{[\tau'_n]}_m\rightarrow \gamma(t)$, where $(\hat u,\gamma)$ is the solution to \eqref{ACobs2} with $\hat u(0) = u_0$. 
\end{thm}
\begin{nb}
By the uniqueness of AC trajectories from Theorem \ref{uniquenessprequel},
%{(}i.e. Corollaries~\ref{ACuniqueness} and \ref{ACuniqueness2}{)}
we can employ a trivial, if lesser known, fact about topological spaces towards removing the need to pass to a subsequence{:} if $(X,\rho)$ is a topological space and $x_n,x\in X$ are such that every subsequence of $x_n$ has a further subsequence converging to $x$, then $x_n\rightarrow x$.\footnote{Suppose $x_n\nrightarrow x$. Then there exists $U \in\rho$ such that $x\in U$ and infinitely many $x_n\notin U$. Choose $x_{n_k}$ such that for all $k$, $x_{n_k}\notin U$. This subsequence has no further subsequence converging to $x$. \qed}\
Let $\tau_n\downarrow 0$ with $\tau_n<\varepsilon$ for all $n${:} define the sequence $x_n:=t\mapsto u^{[\tau_n]}_{\ceil{t/\tau_n}}\in (\V_{t\in[0,\infty)},\rho)$ with $\rho$ the topology of pointwise convergence. Then by the above, every subsequence $x_{n_k}$ has a subsequence converging to an AC solution with initial condition $u_0$. By uniqueness, there is only one such solution $\hat u$, which we can take as the ``$x$'' in the fact. We therefore have that $x_n \rightarrow \hat u$ pointwise, without passing to a subsequence.
Likewise, since the corresponding $\gamma$ is unique up to a.e. equivalence, we have that $z_{\tau_n}\rightharpoonup z$ and $\gamma_{\tau_n}\rightharpoonup\gamma$ in $L^2_{loc}$ without passing to a subsequence.
\end{nb}
\subsection{Consequences of Theorem \ref{SDlimit}}\label{consequences}
To round out this section, we use the representation of the AC solution as a limit of semi-discrete approximations to finish the proof of well-posedness of AC flow, and to obtain a control on the decrease of $\GL$ along trajectories. 

First, we prove well-posedness by recalling the continuity property from Theorem \ref{SDLips}. 
\begin{proof}[Proof of Theorem \ref{wellp}]
Let $u_0,v_0\in\V_{[0,1]}$ define AC trajectories $ u, v$.

Fix $t\geq 0$ and let $m:= \ceil{t/\tau_n}$. By Theorem \ref{SDlimit}, we can choose appropriate $\tau_n\downarrow 0$ such that $u_m^{[\tau_n]}\rightarrow u(t)$ and $v_m^{[\tau_n]}\rightarrow v(t)$ as $n\rightarrow\infty$. Then by \eqref{SDlipeqn}:
\[
||u_m^{[\tau_n]}-v_m^{[\tau_n]}||_\V \leq (1-\tau_n/\varepsilon)^{-m}||u_0-v_0||_\V
\]
and taking $n\rightarrow\infty$ gives \eqref{wpeq}.
\end{proof}

Next, to obtain a control on the behaviour of $\GL(\hat u(t))$ we consider the Lyapunov functional $H$ for the semi-discrete scheme defined in \eqref{Lyap}. Recalling \eqref{SDLyap}, for $u\in\V_{[0,1]}$ we consider a scaling of the Lyapunov functional \[ H_\tau(u) :=\frac{1}{2\tau}H(u) = \GL(u)-\frac{1}{2}\tau \ip{u,Q_\tau u}\]
where $\tau^2 Q_\tau := e^{-\tau\Delta}-I + \tau\Delta$. Note that $Q_\tau$ has eigenvalues $\tau^{-2}(e^{-\tau\lambda_k}-1+\tau\lambda_k)\leq\frac{1}{2}\lambda_k^2$, for $\lambda_k$ the eigenvalues of $\Delta$, and so $|\ip{u,Q_\tau u}|\leq \frac{1}{2}||\Delta||^2||u||_\V^2\leq  \frac{1}{2}||\Delta||^2||\mathbf{1}||_\V^2$. 
It follows that as $\tau\rightarrow 0$, $H_\tau \rightarrow \GL$ uniformly on $\V_{[0,1]}$. We deduce the following result.
\begin{prop} \label{Htauprop}
Let $u_\tau, u\in \V_{[0,1]}$ and $||u_\tau-u||_\V \rightarrow 0$ as $\tau\rightarrow 0$. Then $H_\tau(u_\tau)\rightarrow \GL(u)$. 
\end{prop}
\begin{proof}
We note that it suffices to show that $H_\tau(u_\tau)-H_\tau(u)\rightarrow 0$, since
\[H_\tau(u_\tau) -\GL(u) = H_\tau(u_\tau) -H_\tau(u) +H_\tau(u) -\GL(u).\] Considering \eqref{Lyap} (and recalling that $\lambda:=\tau/\varepsilon$) we get that \[H_\tau(u_\tau)-H_\tau(u) = \frac{1}{2}\left\langle u_\tau -u, \frac{1}{\varepsilon}\left(1-u_\tau-u\right)+(\Delta -\tau Q_\tau)(u_\tau +u)\right\rangle_\V\rightarrow 0\] since the latter term in the inner product is bounded uniformly in $\tau$.
\end{proof}
\begin{thm}
The AC trajectory $\hat u$ defined by \eqref{uhat} has $\GL(\hat u(t))$ monotonically decreasing in $t$. More precisely\emph{:} for all $t > s \geq 0 $, \be \label{GLstep} 
 \GL(\hat u(s)) -  \GL(\hat u(t)) \geq \frac{1}{2(t-s)} \left|\left|\hat u(s) -\hat u(t) \right|\right|_\V^2.
\ee 
\end{thm}
\begin{proof}
Let  $t > s \geq 0 $ and $m:= \ceil{s/\tau_n}$ and $\ell := \ceil{t/\tau_n}$. 
We note a simple fact about inner product spaces:\footnote{To verify this, simply expand the $||\cdot||_\V^2$ terms as inner products and collect terms.} for all sequences $v_n\in \V$,  \be \label{Plaw} \sum_{n=1}^N  \left|\left|v_n \right|\right|_\V^2  = \frac{1}{N} \left|\left|\sum_{n=1}^N v_n \right|\right|_\V^2 + \frac{1}{N}\sum_{k<n} \left|\left|v_n - v_k \right|\right|_\V^2   \geq \frac{1}{N} \left|\left|\sum_{n=1}^N v_n \right|\right|_\V^2 .%\tag{$*$} 
\ee
Now by \eqref{uhat}, we have $u_m^{[\tau_n]}\rightarrow \hat u(s)$ and $u_\ell^{[\tau_n]}\rightarrow \hat u(t)$. It follows that:
\begin{align*}
 \GL(\hat u(s)) -  \GL(\hat u(t)) &=  \lim_{n\rightarrow\infty} H_{\tau_n}\left(u_m^{[\tau_n]}\right) -  H_{\tau_n}\left(u_\ell^{[\tau_n]}\right)&&\text{ by Proposition \ref{Htauprop}}&\\
&\geq \lim_{n\rightarrow\infty} \frac{1}{2\tau_n}\left(1-\frac{\tau_n}{\varepsilon}\right) \sum_{k=m}^{\ell-1} \left|\left|u_{k+1}^{[\tau_n]}-u_k^{[\tau_n]}\right|\right|_\V^2 &&\text{ by \eqref{Hstep}}&\\
&\geq \lim_{n\rightarrow\infty}  \frac{1}{2\tau_n}\left(1-\frac{\tau_n}{\varepsilon}\right) \frac{1}{\ell-m}\left|\left|u_{\ell}^{[\tau_n]}-u_m^{[\tau_n]}\right|\right|_\V^2&&\text{ by \eqref{Plaw}}& \\
&=\frac{1}{2(t-s)} \left|\left|\hat u(s) -\hat u(t) \right|\right|_\V^2 \geq 0. &&&
\end{align*}
%as desired.
\end{proof}
\begin{nb}
Since $\GL(\hat u(s))-\GL(\hat u(t))\leq \GL(\hat u(s))\leq\GL(\hat u(0))$ it follows by \eqref{GLstep} that \[ \left|\left|\hat u(s) -\hat u(t) \right|\right|_\V\leq \sqrt{|t-s|}\sqrt{2\GL(\hat u(0))}\] which is an explicit $C^{0,1/2}$ condition for $\hat u$. 
\end{nb}
\section{Towards a link to mean curvature flow}\label{MCFsec}
Following \cite{vGGOB}, we wish to explore if the well-known continuum links between the MBO scheme, AC flow and mean curvature flow (MCF) extend to the graph setting. Towards this, in this section we first prove some relevant $\Gamma$-convergence results.\footnote{For detail on $\Gamma$-convergence, see e.g. \cite{Braides} and \cite{DalMaso}. } We then discuss remaining questions for future work on such a link. In particular, we consider the prospects of translating to the graph setting the proof of Chen and Elliott \cite{CE} of convergence of continuum double-obstacle AC flow to MCF.
\subsection{$\Gamma$-convergence results}
A positive answer to the question of linking the graph MBO scheme, AC flow and MCF has been suggested by $\Gamma$-convergence results linking the associated energies of graph AC flow \cite{vGB} and the MBO scheme \cite{OKMBO} to graph \emph{total variation} \[\TV(u) := \frac{1}{2}\sum_{i,j\in\V} \omega_{ij}|u_i-u_j|\] of which graph MCF is a type of descending flow.\footnote{See section \ref{MCFlink} for detail on definitions of graph MCF and their relation to TV.} We show analogous $\Gamma$-convergence results for the new functionals defined in this paper.  
Define the function on $\V_{[0,1]}$:
\[
f_0(u) :=\begin{cases} \frac{1}{2}\TV(u), &u\in\V_{[0,1]}\cap\V_{\{0,1\}},\\ 
\infty, &u\in\V_{[0,1]}\setminus \V_{\{0,1\}}.\end{cases}
\]
Then we have the following $\Gamma$-convergences.
\begin{thm}[Cf. $\text{\cite[Theorem 3.1]{vGB}}$]
The Ginzburg\textendash Landau functional $\GL$ with double-obstacle potential defined in \eqref{GL} 
	%\[f_\varepsilon(u):=\frac{1}{2}||\nabla u||_\mathcal{E}^2+\frac{1}{\varepsilon}\sum_{i\in V} u_i^2(u_i-1)^2 \]
	has $\Gamma$-limit in $\V_{[0,1]}$\emph{:}
	\[ 
	\Glim_{\varepsilon\downarrow 0}\GL = f_0.
	\]
\end{thm} 
\begin{proof} The proof is more or less identical to its counterpart in \cite{vGB}. Let $u_\varepsilon\rightarrow u$ for $u_\varepsilon,u\in\V_{[0,1]}$. Suppose $u_i\in(0,1)$ for some $i\in V$, then eventually $(u_\varepsilon)_i\in (0,1)$ and $\GL(u_\varepsilon)\geq \frac{1}{2\varepsilon}d_i^{r} (u_\varepsilon)_i(1-(u_\varepsilon)_i) \rightarrow \infty$, so $f_0(u)\leq \liminf_{\varepsilon\rightarrow 0} \GL(u_\varepsilon)$. Now if $u\in\V_{\{0,1\}}$ then $f_0(u) = \frac{1}{2}\left|\left|\nabla u\right|\right|_\V^2 = \lim_{\varepsilon\rightarrow 0} \frac{1}{2}\left|\left|\nabla u_\varepsilon\right|\right|_\V^2\leq \liminf_{\varepsilon\rightarrow 0} \GL(u_\varepsilon)$. 

Now let $u\in\V_{[0,1]}$ and choose the recovery sequence $\bar u_\epsilon \equiv u$. If  $u_i\in(0,1)$ for some $i\in V$, then $\GL(u)\geq \frac{1}{2\varepsilon}d_i^{r} u_i(1-u_i) \rightarrow \infty$ so $f_0(u) = \lim_{\varepsilon\rightarrow 0} \GL(u)$. If  $u\in\V_{\{0,1\}}$ then $\GL(u) =  \frac{1}{2}\left|\left|\nabla u\right|\right|_\V^2 = f_0(u)$ so again  $f_0(u) = \lim_{\varepsilon\rightarrow 0} \GL(u)$.
\end{proof}
\begin{cor} \label{LyapGamma}The Lyapunov functional $H$ for the semi-discrete flow defined in \eqref{Lyap} has $\Gamma$-convergence in $\V_{[0,1]}$\emph{:}
	\[ 	\Glim_{\varepsilon\downarrow 0, \: 0< \tau\leq\varepsilon} \frac{1}{2\tau}H = f_0.
\]
\end{cor}
	\begin{proof}
		Recall from \eqref{SDLyap} that \[ 	\frac{1}{2\tau}H(u) = \GL(u) - \frac{1}{2}\tau\ip{u,Q_\tau u}\] where $Q_\tau = \tau^{-2}(e^{-\tau\Delta}-I+\tau\Delta)$.
Now $Q_\tau$ has eigenvalues $\tau^{-2}(e^{-\tau\lambda_k}-1+\tau\lambda_k)\leq\frac{1}{2}\lambda_k^2$, for $\lambda_k$ the eigenvalues of $\Delta$, and so $\ip{u,Q_\tau u}\leq \frac{1}{2}||\Delta||^2||u||_\V^2$  is bounded in $\V_{[0,1]}$ uniformly in $\tau$. It follows that if $\varepsilon_j\downarrow 0$, $0< \tau_j\leq\varepsilon_j$ and $ u_j \rightarrow u$ in $\V_{[0,1]}$ then $\tau_j\ip{u_j,Q_{\tau_j} u_j}\rightarrow 0$. 

We derive the $\Gamma$-convergence of $\frac{1}{2\tau}H$ from that of $\GL$. Consider the lim-inf inequality:
\begin{equation*}\begin{split}
	f_0(u) &\leq \liminf_j \text{GL}_{\varepsilon_j}(u_j) \:\:\:\text{ (by the $\Gamma$-convergence of $\GL$ to $f_0$)}\\
		   &=    \liminf_j \text{GL}_{\varepsilon_j}(u_j) + \liminf_j \left(- \frac{1}{2}\tau_j\ip{ u_j,Q_{\tau_j} u_j}\right) \leq \liminf_j \frac{1}{2\tau_j}H(u_j).
\end{split} \end{equation*}
Next take $\bar u_j \rightarrow u$ a recovery sequence for $\GL$ to $f_0$:
\[f_0(u) =\lim_{j\rightarrow\infty} \text{GL}_{\varepsilon_j}(\bar u_j) = \lim_{j\rightarrow\infty} \left(\text{GL}_{\varepsilon_j}( \bar u_j) - \frac{1}{2}\tau_j\ip{ \bar u_j,Q_{\tau_j} \bar u_j}\right)=\lim_{j\rightarrow\infty} \frac{1}{2\tau_j}H(\bar u_j) \] and this proves the  $\Gamma$-convergence of $\frac{1}{2\tau}H$ to $f_0$.
	\end{proof}
\begin{nb}
Taking $\tau =\varepsilon$ and considering $J(u) := \ip{1-u,e^{-\tau\Delta}u}$, the Lyapunov functional for the MBO scheme \emph{(}see \emph{\cite[Proposition 4.6]{vGGOB})}, we have that $H = J$ and so in $\V_{[0,1]}$\emph{:}
\begin{equation*}
	\Glim_{\tau\downarrow 0} \frac{1}{\tau}J|_{\V_{[0,1]}} = 2f_0.
\end{equation*}
This is a special case of the result of \emph{\cite[Theorem 5.10]{OKMBO}}.
\end{nb}
\subsection{Towards a rigorous link}\label{MCFlink} We now discuss some directions for future research towards rigorously linking graph MCF to the MBO scheme and AC flow.
 
A natural research direction would be to translate the continuum proofs of the links to MCF into the graph setting. In \cite{CE}, Chen and Elliott prove that continuum double obstacle Allen--Cahn flow converges to MCF. Encouragingly, the key lemmas for this proof are a pair of comparison principles that in this paper we translate to the graph setting as Theorems \ref{cp2} and \ref{cp1}. Briefly, Chen and Elliott's proof runs by considering double-obstacle AC flow with $u(x,0) = 1$ ``inside" a manifold $\Gamma(0)$ and $u(x,0)=0$ ``outside" $\Gamma(0)$, except for an $\bigO(\varepsilon)$ interfacial region. Then\break \cite[Theorem 3.1]{CE} states that $u(x,t) = 1$ inside $\Gamma(t)$ and $0$ outside $\Gamma(t)$, except for an $\bigO(\varepsilon)$ interfacial region, where $\Gamma(t)$ is the MCF evolution of $\Gamma(0)$. They then extend this result to more general initial conditions. However, this proof does not obviously translate to the graph context. In particular, \cite{CE} uses facts about the signed distance function $d(x,t):=\text{sdist}(x,\Gamma(t))$ such as \begin{align*}&|\nabla d(x,t)|=  1 \text { a.e.}&  \text{and}&  &d_t -\Delta d = 0 \text{ on }\Gamma(t)&  \end{align*} which do not have obviously useful analogues in the graph setting. A topic for future research will be if \cite[Theorem 3.1]{CE} can be translated to the graph setting using Theorems \ref{cp2} and \ref{cp1}.

A further complication arises in how exactly to define MCF on a graph. In \cite{vGGOB}, motivated by the variational formulation of MCF in \cite{ATW}, graph MCF was defined as the minimisation scheme
\[
S_{n+1} \in \argmin_{S\subseteq V} \TV(\chi_S) + \frac{1}{\delta t}\ip{\chi_S-\chi_{S_n},(\chi_S-\chi_{S_n})d^{\Sigma_n}}
\]
where $\Sigma_n :=\{ i\in V\mid \exists j\in V\text{ s.t. }\omega_{ij}>0\text{ and }\chi_{S_n}(i)\neq \chi_{S_n}(j)\}$ is the \emph{boundary} of $S_n$, and $d_i^{\Sigma_n}$ is the graph distance from $i$ to $\Sigma_n$. But it is relatively easy to come up with graphs upon which this will behave very differently to any diffusion-based flow. Let $G=(V,E)$ be a graph as in section \ref{groundwork}, and for $\epsilon>0$ let $G^\epsilon=(V,V^2)$ be the complete graph defined by $\omega^\epsilon_{ij} := \omega_{ij}$ if $ij\in E$ and $\omega^\epsilon_{ij} := \epsilon$ otherwise. Then $\Delta^\epsilon =\Delta +\bigO(\epsilon)$, so for $\epsilon$ sufficiently small the AC flow and MBO scheme will be essentially the same on $G$ and $G^\epsilon$. But if $S_n\notin\{\emptyset,V\}$, then by definition $\Sigma_n^\epsilon = V$ and so $d_i^{\Sigma^\epsilon_n}=\mathbf{0}$, so on $G^\epsilon$ we only get solutions $S_{n+1}^\epsilon \in \{\emptyset,V\}$ which in general will be very different from the behaviour on $G$. 

Towards resolving this issue, we recall the fact from \cite{vGGOB} that %the Lyapunov functional $J$ for the MBO scheme has the property that  
\[
J(\chi_S) =\ip{\mathbf{1}-\chi_S,e^{-\tau\Delta}\chi_S} = \tau \TV(\chi_S) + R_S(\tau)
\]
where 
$R_S(\tau):= \sum_{k\geq 2} (-1)^k \frac{\tau^k}{k!} \ip{\chi_{S^c}, \Delta^k \chi_S}.$
Then the MBO update of $S_n\subseteq V$ obeys
\begin{align*}
S_{n+1} \in \argmin_{S\subseteq V}\: \,&\ip{\mathbf{1}-2e^{-\tau\Delta}\chi_{S_n},\chi_S} \\
&= J(\chi_S) + \ip{\chi_S-\chi_{S_n},e^{-\tau\Delta}(\chi_S-\chi_{S_n})} - \ip{\chi_{S_n},e^{-\tau\Delta}\chi_{S_n}}\\
&\simeq \TV(\chi_S) + \frac{R_S(\tau)+||e^{-\frac{1}{2}\tau\Delta}(\chi_S-\chi_{S_n})||_\V^2}{\tau}.
\end{align*}
Since $R_S(\tau)/\tau = \bigO(\tau)$, this suggests that the following definition of graph MCF:
\be\label{newMCF}
S_{n+1} \in \argmin_{S\subseteq V} \TV(\chi_S) + \frac{||e^{-\frac{1}{2}\tau \Delta}(\chi_S-\chi_{S_n})||_\V^2}{\tau}
\ee
might strongly resemble the MBO scheme. We note immediately that \eqref{newMCF} avoids the problem described in the previous paragraph. A topic for future research will be verifying that \eqref{newMCF} does indeed resemble the MBO scheme, and that \eqref{newMCF} deserves to be called a mean curvature flow.
\begin{nb} We note briefly an alternative approach to defining MCF on a graph, that of Elmoataz and co-authors, which has thus far fallen beyond the scope of our research. Motivated by the well-known level-set formulation of continuum MCF, in \emph{\cite{Elmo2014}} El Chakik, Elmoataz, and Desquesnes define the graph mean curvature\footnote{For $u=\chi_S$ we note a similarity between this definition and \cite[Definition 3.2]{vGGOB}.} and up/downwind gradient norms of $u\in \V$ by\footnote{We adapt their notation to align with that of this paper. Recall that $(\nabla u)_{ij} := u_j - u_i$ for $i$ and $j$ neighbours.} 
\begin{align*}
	&K_i(u) := \sum_{j\in V} \frac{\omega_{ij}}{d_i} \sgn{(\nabla u)}_{ij}, & \hspace{-0.7em}\left|\left|(\nabla^{\pm}u)_i\right|\right|^p_p :=\sum_{j\in V} \omega_{ij}\left|{(\nabla u)}_{ij}^\pm\right|^p, & \left|\left|(\nabla^{\pm}u)_i\right|\right|_\infty := \max_{j\in V} \omega_{ij} \left| {(\nabla u)}_{ij}^\pm\right|,
\end{align*}
where superscript $\pm$ denote positive/negative parts, and define graph MCF as the ODE flow:
\be\label{ElmoMCF}
\frac{du_i}{dt} = K_i^+(u(t))\left|\left|(\nabla^{+}u(t))_i\right|\right|_p - K_i^-(u(t))\left|\left|(\nabla^{-}u(t))_{i}\right|\right|_p
\ee
 for $p\in [1,\infty]$. We note that, for $r=1$, it is easy to see that $-K(u)$ is the first variation of TV at $u$, and that therefore \eqref{ElmoMCF} monotonically decreases TV along trajectories.\footnote{We leave the details of this demonstration as an exercise for the reader.} A topic for future research will be whether \eqref{ElmoMCF} describes the same dynamics as schemes like \eqref{newMCF}, and if this ODE perspective is more fruitful for linking to AC flow or to the MBO scheme.
\end{nb}

\section{Conclusions}

Following Blowey and Elliott \cite{BE1991,BE1992,BE1993}, we have defined a graph Allen--Cahn equation using a double-obstacle potential, and proved well-posedness and Lipschitz regularity of solutions, and demonstrated that the graph MBO scheme is a special case of a ``semi-discrete'' scheme for this equation. We exhibited a Lyapunov functional for this scheme, and via this performed an analysis of the long-time behaviour similar to Luo and Bertozzi \cite{LB2016}. Furthermore, we proved that for any $\tau_n\downarrow 0$ there is a subsequence along which the semi-discrete iterates $(u_m,\beta_m)$, with initial state $u_0$ and step $\tau_n$, converge a.e. to the Allen--Cahn trajectory with $u(0)=u_0$, and in particular the $u_m$ converge pointwise without passing to a subsequence.

 Towards a link to mean curvature flow, we have proved some promising $\Gamma$-convergences, and translated two comparison principles\textemdash used by Chen and Elliott to show convergence of continuum Allen--Cahn flow to mean curvature flow in \cite{CE}\textemdash into the graph setting. We have also discussed some topics for future work in this direction.%The second of these principles furthermore proved uniqueness of graph Allen--Cahn trajectories.

In future work we will continue investigating links to mean curvature flow, in particular seeking a representation of mean curvature flow on a graph that will allow us to derive a concrete link between Allen--Cahn/MBO and mean curvature flow on graphs. Currently ongoing work by the authors investigates the semi-discrete link between the Allen--Cahn flow and MBO scheme under further constraints, in particular \cite{BuddvG} and \cite{Buddfidelity} showing respectively that the link continues to hold under mass-conservation and fidelity-forcing constraints. We furthermore suspect that the semi-discrete link also holds in the continuum, which may prove a topic for future research.

\section*{Acknowledgements}

%This project has received funding from the European Union’s Horizon 2020 research and innovation programme under the Marie Skłodowska-Curie grant agreement No 777826.

Much of the work for this project was performed whilst both authors were employed by the University of Nottingham.

\appendix

\section{Proof of Theorem \ref{existence}}\label{existsec}
Recall that given initial condition $u(0) = u_0\in\V_{[0,1]}$, we seek to find  $(u,\beta)\in\V_{[0,1],t\in [0,\infty)}\times\V_{t\in  [0,\infty)}$ with $u\in H_{loc}^1( [0,\infty);\V)\cap C^0( [0,\infty);\V)$ that solves for a.e. $t\geq 0$
\begin{align}
\label{ACobs2aa}
	& \frac{du(t)}{dt} = - \Delta u(t) + \frac{1}{\varepsilon}\left(u(t) -\frac{1}{2}\mathbf{1}\right) +\frac{1}{\varepsilon} \beta(t), &\beta(t)\in\mathcal{B}(u(t)).
\end{align}
Let $\nu>0$. Define the following $C^1$ approximation to the double obstacle potential 
\be W_\nu(x) = \begin{cases}\frac{1}{4\nu}x^2 +\frac{1}{2}x, & x<0\\
				\frac{1}{2}x(1-x), & 0\leq x \leq 1 \\
				\frac{1}{4\nu}(x-1)^2 -\frac{1}{2}(x-1), &x>1
	\end{cases} \ee
with derivative 
\be \label{Wnudiff} W'_\nu(x) = \begin{cases}\frac{1}{2\nu}x +\frac{1}{2}, & x<0,\\
				\frac{1}{2}-x, & 0\leq x \leq 1, \\
				\frac{1}{2\nu}(x-1) -\frac{1}{2}, &x>1.
	\end{cases} \ee
Note that $W_\nu$ is a double-well potential with wells of depth $-\nu/4$ at $-\nu$ and $1+\nu$. Then we define the corresponding AC: 
\be
\label{ACnu} 
\frac{du_\nu(t)}{dt} = - \Delta u_\nu(t) - \frac{1}{\varepsilon}W'_\nu\circ u_\nu(t)
\ee
with $u_\nu(0) = u_0\in\V_{[0,1]}$. We note the following facts.
\begin{prop} For $u_\nu$ solving \eqref{ACnu} with with $u_\nu(0) \in\V_{[0,1]}$\emph{:}
\begin{enumerate}
\item $u_\nu \in C^1( [0,\infty);\V)$ exists.
\item $u_\nu \in \V_{[-\nu,1+\nu],t\in[0,\infty)}$.
\end{enumerate}
\end{prop}
\begin{proof} We employ standard arguments.
\begin{enumerate}
\item $W'_\nu$ is piecewise linear, so is Lipschitz. Hence the RHS of \eqref{ACnu} is Lipschitz in $u_\nu$, so existence of a $C^1$ solution on $[0,\infty)$ follows by the Picard--Lindel\"{o}f Theorem. 
\item %Note that $u_0\in \V_{[-\nu,1+\nu]}$ and $-\nu,\:1+\nu$ are the locations of the wells of $W_\nu$. 
Suppose $\exists T>0, k \in V$ such that $(u_\nu(T))_k<-\nu$. On $[0,T]$ each $(u_\nu)_i$ is continuous so attains its lower bound. As $V$ is finite we may choose the $i\in V$ with lowest such bound, which by assumption is less than $-\nu$, and let $t\in [0,T]$ be a time this bound is attained. Then $(u_\nu(t))_i<-\nu$, and for all $j\in V$, $(u_\nu(t))_j\geq(u_\nu(t))_i$, so 
\[\frac{d(u_\nu)_i}{dt}(t) = - (\Delta u_\nu(t))_i - \frac{1}{\varepsilon}W'_\nu((u_\nu(t))_i)> - (\Delta u_\nu(t))_i \geq 0\] since $(\Delta u_\nu(t))_i \leq 0$ when $i$ minimises $u_\nu(t)$. If $t>0$ then we have $0<t'<t$ such that \[\frac{(u_\nu(t))_i-(u_\nu(t'))_i}{t-t'}\geq \frac{1}{2}\frac{d(u_\nu)_i}{dt}(t) >0 \] so $(u_\nu(t'))_i<(u_\nu(t))_i$ contradicting the minimality of $t$. If $t=0$, $(u_\nu(t))_i=(u_0)_i\geq-\nu$. So we get a contradiction in either case. Likewise for the upper bound.
\end{enumerate}%\vspace*{-2.0em}
\end{proof}
\begin{lem}
If $u\in C^{0,1}([0,\infty),\V)$ then $u\in  H_{loc}^1( [0,\infty);\V)$. 
\end{lem}
\begin{proof}
As $u$ is Lipschitz, for all $i\in V$ we have $u_i\in C^{0,1}([0,\infty);\mathbb{R})$, so by \cite[Theorem 7.13]{Leoni} $u_i\in H^1_{loc}([0,\infty);\mathbb{R})$. Then we recall from Proposition \ref{graphH1} that $u\in H^1_{loc}(T;\V)$ if and only if for all $i\in V$, $u_i\in H^1_{loc}(T;\mathbb{R})$. 
\end{proof}
We next demonstrate the following convergences.
\begin{lem}\label{convergenu} For any sequence $\hat\nu_n\downarrow 0$ there exists a subsequence $\nu_n$ and a Lipschitz function $u\in\V_{[0,1],t\in [0,\infty)}\cap H^1_{loc}( [0,\infty);\V)$ such that for all compact intervals $T \subseteq [0,\infty)$, $u_{\nu_n}|_T \rightarrow u|_T$ uniformly, where $u_{\nu_n}$ are $C^1$ solutions to \eqref{ACnu}. Furthermore the derivatives $\frac{du_{\nu_n}}{dt}$ converge weakly to $ \frac{du}{dt}$  in $L^2_{loc}$, where $\frac{du}{dt}$ is the weak derivative of $u$.
\end{lem}
\begin{proof} We use the Arzela--Ascoli theorem. WLOG we may take $\hat\nu_n\leq 1$ and so we have $u_{\hat\nu_n} \in \V_{[-1,2],t\in[0,\infty)}$, thus the $u_{\hat\nu_n}$ are uniformly bounded in $t$. Next, $\Delta$ is a bounded linear map so the $\Delta u_{\hat\nu_n}$ are uniformly bounded in $n$ and $t$, and from \eqref{Wnudiff} $W'_\nu\circ u_{\hat\nu_n}(t)\in\V_{[-1/2,1/2]}$. Hence by \eqref{ACnu} $\frac{du_{\hat\nu_n}}{dt}$ is uniformly bounded in $n$ and $t$, and so for any $t_1\leq t_2$,
\be\label{equiLips}||u_{\hat\nu_n}(t_1)-u_{\hat\nu_n}(t_2)||_\V \leq \int_{t_1}^{t_2} \left|\left|\frac{du_{\hat\nu_n}}{dt}(t)\right|\right|_\V\; dt\leq C|t_1-t_2|\ee where $C$ is independent of $n$, $t_1$ and $t_2$. Hence the $u_{\hat\nu_n}$ are uniformly bounded and equicontinuous on every compact interval. 

We define $\nu_n$ by a diagonal argument. Take $T=[0,1]$, then by  Arzela--Ascoli we have some $\nu_n^{(1)}$, a subsequence of $\hat\nu_n$, such that $u_{\nu^{(1)}_n}|_{[0,1]}$ converges uniformly. We then iterate: take $T=[0,k+1]$, then by  Arzela--Ascoli we have some $\nu_n^{(k+1)}$, a subsequence of $\nu_n^{(k)}$, such that $u_{\nu^{(k+1)}_n}|_{[0,k+1]}$ converges uniformly. Finally, let $\nu_n := \nu_n^{(n)}$. Then $\nu_n$ is eventually a subsequence of each $\nu_n^{(k)}$, so we have $u:[0,\infty)\rightarrow\V$ well-defined by: $u|_{[0,N]}:=$ the uniform limit of $u_{\nu_n}|_{[0,N]}$, for all $N\in\mathbb{N}$. For all $i\in V$, $t\in [0,\infty)$ and $n\in\mathbb{N}$ we have $(u_{\nu_n}(t))_i \in[-\nu_n, 1+\nu_n]$, and so taking $n\rightarrow \infty$, $u_i(t)\in[0, 1]$. Thus $u\in \V_{[0,1],t\in[0,\infty)}$. 
%Finally, we show regularity of $u$ and weak convergence of $\frac{du_{\nu_n}}{dt}$ up to a subsequence of $\nu_n$. 
Taking limits as $n\rightarrow\infty$ in \eqref{equiLips}: 
\[\forall t,s\geq 0 \:\:\: \left|\left|u(t)-u(s)\right|\right|_\V \leq C|t-s|\]
hence $u$ is Lipschitz on $[0,\infty)$ and therefore is $H^1_{loc}([0,\infty);\V)$ by the above lemma. 

Let $T\subseteq[0,\infty)$ be any bounded interval. As the $\frac{du_{\nu_n}}{dt}$ are uniformly bounded in $n$ and $t$ they lie in a closed bounded ball in $L^2(T;\V)$, which by the Banach--Alaoglu theorem is weak-compact. So $\frac{du_{\nu_n}}{dt}$ weakly converges on $T$ up to a subsequence of $\nu_n$. By a diagonal argument as above, covering $[0,\infty)$ by countably many such $T$s, we get weak convergence on $[0,\infty)$ up to a subsequence of $\nu_n$.  We redefine $\nu_n$ to be this new subsequence. 

Denote the weak limit of  $\frac{du_{\nu_n}}{dt}$ by $g$. %As the $\frac{du_{\nu_n}}{dt}$ are uniformly bounded, $g$ is a bounded function. 
Then for any open bounded $T$  and $\varphi\in C^\infty_c(T;\V)$ 
\[(g,\varphi)_{t\in T} = \lim_{n\rightarrow \infty}\left(\frac{du_{\nu_n}}{dt},\varphi\right)_{t\in T} = -\lim_{n\rightarrow \infty}\left(u_{\nu_n},\frac{d\varphi}{dt}\right)_{t\in T} =- \left(u,\frac{d\varphi}{dt}\right)_{t\in T} \] so $g=\frac{du}{dt}$, where the first equality comes from weak convergence of  $\frac{du_{\nu_n}}{dt}$ and the final equality from uniform convergence of $u_{\nu_n}$ to $u$ on $\bar T$.
\end{proof}
\begin{thm}
For $u\in C^{0,1}([0,\infty);\V_{[0,1]})\cap H^1_{loc}( [0,\infty);\V)$ as in the previous lemma and almost every $t\in[0,\infty)$, there exists $\beta(t) \in \mathcal{B}(u(t))$ such that \[\frac{du}{dt} = - \Delta u(t) + \frac{1}{\varepsilon}\left(u(t) -\frac{1}{2}\mathbf{1}\right) +\frac{1}{\varepsilon} \beta(t).\]
\end{thm}
\begin{proof} 
Take $\hat\nu_n\downarrow 0$ with subsequence $\nu_n$ and corresponding $u_{\nu_n}$ as in the previous lemma. Define $\beta_{\nu} := \frac{1}{2}\mathbf{1}-u_\nu -W'_\nu\circ u_\nu$. Then it is easy to check that 
 \[(\beta_{\nu_n}(t)) _i=\begin{cases}-\left(1 + \frac{1}{2\nu_n}\right)(u_{\nu_n}(t))_i, &(u_{\nu_n}(t))_i \in [-\nu_n,0]\\
						        0, &  (u_{\nu_n}(t))_i\in[0,1]\\ 
						        \left (1+\frac{1}{2\nu_n}\right)(1-(u_{\nu_n}(t))_i), &(u_{\nu_n}(t))_i \in [1,1+\nu_n]
\end{cases}\]
and by \eqref{ACnu}, \[\frac{du_{\nu_n}}{dt}=-\Delta u_{\nu_n}(t) +\frac{1}{\varepsilon}\left(u_{\nu_n}(t)-\frac{1}{2}\mathbf{1}\right)+\frac{1}{\varepsilon}\beta_{\nu_n}(t).\]
Since $u_{\nu_n}\rightarrow u$ uniformly and $\frac{du_{\nu_n}}{dt}\rightharpoonup \frac{du}{dt}$  (in $L^2_{loc}$), we have that  \[\frac{1}{\varepsilon}\beta_{\nu_n}= \frac{du_{\nu_n}}{dt} + \Delta u_{\nu_n}-\frac{1}{\varepsilon}\left(u_{\nu_n}-\frac{1}{2}\mathbf{1}\right)\rightharpoonup \frac{du}{dt} + \Delta u-\frac{1}{\varepsilon}\left(u-\frac{1}{2}\mathbf{1}\right) =:\frac{1}{\varepsilon}\beta.\]
It suffices then to prove that $\beta$ so defined obeys $\beta(t)\in\mathcal{B}(u(t))$ for a.e. $t\geq 0$. By the Banach--Saks theorem \cite{BS}, passing to a further subsequence of $\nu_n$, the Cesàro sums converge strongly, i.e. for all bounded intervals $T$ \[ \frac{1}{N}\sum_{n=1}^N \beta_{\nu_n} \rightarrow \beta \text { in } L^2(T;\V)\] as $N\rightarrow\infty$.
Next, recall that $L^2$ convergence implies a.e. pointwise convergence of a subsequence. Thus, covering $[0,\infty)$ by countably many such $T$, by a diagonal argument (as in the above lemma) we extract $N_k\rightarrow \infty$ such that for a.e. $t\geq0$ 
\be\label{betanu} \beta(t) = \lim_{k\rightarrow\infty}  \frac{1}{N_k}\sum_{n=1}^{N_k} \beta_{\nu_n}(t). \ee
Fix a $t$ at which this convergence holds. First, suppose $u_i(t) \in (0,1)$. Then eventually $(u_{\nu_n}(t))_i\in(0,1)$ and so $(\beta_{\nu_n}(t)) _i=0$. Therefore by \eqref{betanu}, $\beta_i(t) = 0$ as desired. 

Next, suppose that $u_i(t)=0$. Then eventually $(u_{\nu_n}(t))_i\in[-\nu_n,1)$ and note that for $(u_{\nu_n}(t))_i < 1 $ we have $(\beta_{\nu_n}(t))_i \in [0,\frac{1}{2} + \nu_n]$. It follows by \eqref{betanu} that $\beta_i(t)\geq 0$, as desired. Likewise for $u_i(t) = 1$,  $\beta_i(t) \leq 0$. Therefore $\beta(t)\in\mathcal{B}(u(t))$, completing the proof.
\end{proof}
\begin{nb}As an example, let $u_0 = \mathbf{0}$. One can check that \eqref{ACnu} has unique solution: \[u_\nu(t) = -\nu \left(1-e^{-\frac{t}{2\varepsilon\nu}}\right)\mathbf{1}.\] Hence we have the expression for $\beta_\nu$ \[\beta_\nu(t) = \left(\nu +\frac{1}{2}\right)\left(1-e^{-\frac{t}{2\varepsilon\nu}}\right)\mathbf{1}\] and taking $\nu\downarrow 0$ we therefore get the expected AC solution:
\begin{align*}
&u(t)=\mathbf{0}, &\beta(t) = \begin{cases} \mathbf{0}, &t = 0,\\
\frac{1}{2}\mathbf{1}, &t>0.
\end{cases}\end{align*}
\end{nb}
\begin{nb}
We can eliminate some of the need to pass to a subsequence by the uniqueness of AC trajectories {(}i.e. by Theorem \ref{uniquenessprequel}{)} and the standard result{:} if $(X,\rho)$ is a topological space and $x_n,x\in X$ are such that every subsequence of $x_n$ has a further subsequence converging to $x$, then $x_n\rightarrow x$.\footnote{Suppose $x_n\nrightarrow x$. Then there exists $U \in\rho$ such that $x\in U$ and infinitely many $x_n\notin U$. Choose $x_{n_k}$ such that for all $k$, $x_{n_k}\notin U$. This subsequence has no further subsequence converging to $x$. \qed}

Let $\tilde\nu_n\downarrow 0$. By above every subsequence $u_{\hat\nu_{n}}$ of $u_{\tilde\nu_n}$ has a further subsequence $u_{\nu_n}$ uniformly converging on every compact interval to an AC solution $u$ with $u(0)=u_0$. By uniqueness, there is only one such solution $u$, so we take $x$ as this $u$ and $(X,\rho)$ as $\V_{t\in[0,\infty)}$ with the topology of ``uniform convergence on every compact interval" in the result. Hence $u_{\tilde\nu_n} \rightarrow u$ uniformly on every compact interval.
Likewise, as $du/dt$ and $\beta$ are unique up to a.e. equivalence, we have $du_{\tilde\nu_n}/dt \rightharpoonup du/dt$ and $\beta_{\tilde\nu_n}\rightharpoonup\beta$ in $L^2_{loc}$.
\end{nb}
\section{Proof of Theorem \ref{uniquenessprequel}}\label{uniquesec}
We first note a useful pair of facts.
\begin{prop}\label{usefullem}
	Let $z\in\V_{t\in T}$ for $T$ any interval and let $z_{+}(t)$ be the positive part of $z(t)$, i.e. $(z_+)_i(t) :=\max\{z_i(t),0\}$. Then for all $ t\in T$, 
	\[\langle \nabla z_+(t),\nabla z_+(t)\rangle_\mathcal{E}\leq \langle \nabla z(t),\nabla z_+(t)\rangle_\mathcal{E}.\] 
	Also if $z \in H^1(T;\V)\cap C^0(\bar T;\V)$ and $T=(0,T^*)$ for $T^*>0$ then, \[\frac{1}{2}||z_+(T^*)||_\V^2 -\frac{1}{2}||z_+(0)||_\V^2 =\int_0^{T^*}  \left\langle\frac{dz}{dt},z_+\right\rangle_\V \; dt .\]
\end{prop}
\begin{proof}
	Consider $\langle\nabla z -\nabla z_+,\nabla z_+\rangle_\mathcal{E}$. By definition we have \[\langle\nabla z -\nabla z_+,\nabla z_+\rangle_\mathcal{E}=\frac{1}{2}\sum_{i,j\in V} \omega_{ij}X_{ij}\] where $X_{ij}:=((z_+)_i -(z_+)_j)(z_i-(z_+)_i -z_j +(z_+)_j)$. We claim that $X_{ij}\geq 0$. WLOG suppose $(z_+)_i \geq (z_+)_j$. If $(z_+)_i = (z_+)_j$ then $X_{ij}=0$, and if $z_i,z_j > 0$ then $X_{ij}=0$. Finally if $z_i > 0, z_j \leq 0$ then $X_{ij}=-z_iz_j\ge 0$. So $\langle\nabla z -\nabla z_+,\nabla z_+\rangle_\mathcal{E}\geq 0$. 
	
For the latter claim, note that it suffices to show that for each $i\in V$ \be\label{fact}\frac{1}{2}(z_+(T^*))_i^2 -\frac{1}{2}(z_+(0))_i^2  =\int_0^{T^*} \frac{dz_i(t)}{dt}(z_+)_i(t) \; dt \ee and recall that $z \in H^1(T;\V)$ if and only if $z_i\in H^1(T;\mathbb{R})$ for each $i\in V$. The equation \eqref{fact} then follows from \cite[Lemma 3.3]{H1reg}. \end{proof}
We now prove a slight strengthening of a graph analogue of \cite[Lemma 2.4]{CE}. 
\begin{thm}[Cf. $\text{\cite[Lemma 2.4]{CE}}$]\label{cp2}
	Let $u,v\in H^1((0,T);\V)\cap C^0([0,T];\V)$ and $\beta,\gamma\in\V_{t\in[0,T]}$ be such that for all $i\in V$, for all $t\in (0,T)$ $0\leq u_i(t),v_i(t)\leq 1$, and for a.e. $t\in(0,T)$
\begin{align}
\label{comp1}
	&\varepsilon \frac{du_i}{dt} + \varepsilon(\Delta u(t))_i+\frac{1}{2}-u_i(t)\geq \beta_i(t), &\beta(t)\in\mathcal{B}(v(t))\\
\label{comp2}
	&\varepsilon \frac{dv_i}{dt} + \varepsilon(\Delta v(t))_i+\frac{1}{2}-v_i(t)\leq \gamma_i(t), &\gamma(t)\in\mathcal{B}(v(t)).
\end{align}
Then %if $u\in H^1((0,T);\V)\cap C^0([0,T];\V_{[0,1]})$ is a solution to \eqref{ACobs2} and 
for all $i\in V$, $v_i(0)\leq u_i(0)$, it follows that for all $t\in [0,T]$ and $i\in V$, $v_i(t)\leq u_i(t)$.
\end{thm}
\begin{proof}
Subtracting \eqref{comp1} from \eqref{comp2}, we get that for a.e. $t\in (0,T)$ (understanding the inequality vertexwise) \[\varepsilon\frac{d}{dt}(v(t)-u(t)) +\varepsilon\Delta(v(t)-u(t)) -(v(t)-u(t))\leq \gamma(t)-\beta(t) \] where $\beta(t)\in\mathcal{B}(u(t))$. Let $w:=v-u$ and take the inner product with $w_+$
\[\varepsilon\left\langle\frac{dw}{dt},w_+(t)\right\rangle_\V+\varepsilon\ip{\Delta w(t),w_+(t)} -\ip{w(t),w_+(t)} \leq \ip{\gamma(t)-\beta(t),w_+(t)}. \]
Consider the $RHS$. If $v_i(t)\leq u_i(t)$ then $(w_+(t))_i(\gamma_i(t)-\beta_i(t)) = 0$, and if $v_i(t) > u_i(t)$ a simple case check reveals that $\gamma_i(t)\leq \beta_i(t)$. Therefore $RHS\leq 0$. Hence we have that \begin{equation*}\begin{split}\ip{w_+(t),w_+(t)} =\ip{w(t),w_+(t)}
     										   &\geq \varepsilon\left\langle\frac{dw}{dt},w_+(t)\right\rangle_\V+\varepsilon\ip{\Delta w(t),w_+(t)}\\
&= \varepsilon\left\langle\frac{dw}{dt},w_+(t)\right\rangle_\V+\varepsilon\left\langle\nabla w(t),\nabla w_+(t)\right\rangle_\mathcal{E}\\
     										   &\geq \varepsilon\left\langle\frac{dw}{dt},w_+(t)\right\rangle_\V+\varepsilon\left\langle\nabla w_+(t),\nabla w_+(t)\right\rangle_\mathcal{E} \text{by Proposition \ref{usefullem}}\\
     										   &\geq \varepsilon\left\langle\frac{dw}{dt},w_+(t)\right\rangle_\V
\end{split}\end{equation*}
for a.e. $t\in (0,T)$. 
Note that $w_+(0) = \mathbf{0}$. Thus integrating and applying the second part of Proposition \ref{usefullem}, we have that \[\frac{\varepsilon}{2}||w_+(T)||^2_\V \leq \int_0^T ||w_+(t)||^2_\V \; dt\]
so by Gr\"onwall's integral inequality $ ||w_+(t)||^2_\V \leq 0$ for all $t\in [0,T]$, and hence $w_+(t) = \mathbf{0}$ in $[0,T]$. Therefore $v_i(t)\leq u_i(t)$ for all $i\in V$ and $t\in [0,T]$. 
\end{proof}
\begin{nb}
	The condition that $v_i(t)\geq 0$ can somewhat be relaxed. If $v_i(t)<0$ then from the subdifferential $\gamma_i(t) = \infty$, in which case \eqref{comp2} is still meaningfully satisfied. The only hiccup that arises in the proof is the $v_i(t)\leq u_i(t)$ case for $(w_+(t))_i(\gamma_i(t)-\beta_i(t))$, as this becomes the undefined $0\times\infty$. But if we consider $u$, $v$, $\beta$ and $\gamma$ arising as in Appendix \ref{existsec} from a limit of $C^1$ potentials approaching $W$, then for $v_i(t)<0$ the corresponding limiting term \[((v_{\nu_n}(t))_i-(u_{\nu_n}(t))_i)_+\left(\frac{1}{2}-(v_{\nu_n}(t))_i-W'_{\nu_n}((v_{\nu_n}(t))_i) - \frac{1}{2}+(u_{\nu_n}(t))_i+W'_{\nu_n}((u_{\nu_n}(t))_i)\right)\] has eventually $(v_{\nu_n}(t))_i\leq-\nu_n \leq(u_{\nu_n}(t))_i$, so the term is eventually constantly zero. Hence we may take $(w_+(t))_i(\gamma_i(t)-\beta_i(t))=0$ as desired. Likewise $u_i(t)\leq 1$ can be relaxed.
\end{nb}
Finally, this comparison principle yields uniqueness of AC solutions.
\begin{cor} \label{ACuniqueness} 
Let $u,v\in H^1((0,T);\V)\cap C^0([0,T];\V_{[0,1]})$ and $\beta,\gamma\in\V_{t\in[0,T]}$. Let $(u,\beta)$, $(v,\gamma)$ solve \eqref{ACobs2} on $[0,T]$ with $u(0)=v(0)$. Then for all $t\in [0,T]$, $u(t) = v(t)$, and for a.e. $t\in [0,T]$, $\beta(t) = \gamma(t)$.   
\end{cor}
\begin{proof}
Since $(v,\gamma)$ solves \eqref{ACobs2}, it follows immediately that they satisfy \eqref{comp2}. Furthermore, $u(0)\leq v(0)$ vertexwise. Hence by the comparison principle, $u(t)\leq v(t)$ vertexwise for all $t\in[0,T]$. By symmetry, $v(t)\leq u(t)$ vertexwise, and hence $u(t)=v(t)$, for all $t\in [0,T]$. Finally, by Theorem \ref{betathm}, $\beta(t)$ and $\gamma(t)$ are uniquely determined a.e. by $u(t)$ and $v(t)$, and therefore $\beta(t) = \gamma(t)$ for a.e. $t\in [0,T]$. 
\end{proof}
\begin{cor}\label{ACuniqueness2}Let $u,v\in H^1((0,\infty);\V)\cap C^0([0,\infty);\V_{[0,1]})$and $\beta,\gamma\in\V_{t\in[0,\infty)}$. Let $(u,\beta)$, $(v,\gamma)$ solve \eqref{ACobs2} on $[0,\infty)$ with $u(0)=v(0)$.  Then for all $t\in [0,\infty)$, $u(t) = v(t)$, and for a.e. $t\in [0,\infty)$, $\beta(t) = \gamma(t)$.   
\end{cor}
\begin{proof}
For all $T\in [0,\infty)$, $(u|_{[0,T]},\beta|_{[0,T]}),(v|_{[0,T]},\gamma|_{[0,T]})$ are solutions on $[0,T]$ with $u(0)=v(0)$. Hence $u|_{[0,T]}=v|_{[0,T]}$, and $\beta|_{[0,T]} = \gamma|_{[0,T]}$ almost everywhere. Thus for all $t\in [0,\infty)$, $u(t)=v(t)$, and for a.e. $t\in [0,\infty)$, $\beta(t) = \gamma(t)$.
\end{proof}
\section{Another comparison principle from \cite{CE}}\label{compsec}
We here prove a graph analogue of \cite[Lemma 2.3]{CE}.

From Theorem \ref{ACobsweak} we recall the weak formulation of double-obstacle AC flow: if $u\in\V_{[0,1],t\in T}$ a solution to double-obstacle AC flow \eqref{ACobs2} then for a.e. $t\in T$
\begin{equation}
\label{ACobs3}
\forall\eta\in\V_{[0,1]},\: \left\langle\frac{du}{dt}-u(t)+\frac{1}{2\varepsilon}\mathbf{1},\eta-u(t)\right\rangle_\V+\left\langle\nabla u(t),\nabla\eta-\nabla u(t)\right\rangle_\mathcal{E}\geq0.
\end{equation} 
\begin{thm}[Cf. $\text{\cite[Lemma 2.3]{CE}}$]
\label{cp1}
	Let $w\in\V_{(-\infty,1], t\in(0,T)}\cap H^1((0,T);\V)$ be continuous and let $u\in\V_{[0,1],t\in(0,T)}\cap H^1((0,T);\V)$ be continuous and obey \eqref{ACobs3}. Suppose that $w_i(0)\leq u_i(0)$ and that there exists $f\in\V_{t\in(0,T)}$ such that for all $T^*\in[0,T]$\be\label{fcon}\int_0^{T^*}\ip{f(t),(w-u)_{+}(t)}\; dt \leq  \int_0^{T^*} \ip{w(t),(w-u)_{+}(t)}\; dt\ee and
\begin{align}\label{compint} & \forall \eta\in\V_{[0,\infty),t\in(0,T)} &\int_0^T \left\langle\frac{dw}{dt}+\frac{1}{2\varepsilon}\mathbf{1},\eta\right\rangle_\V + \langle \nabla w,\nabla\eta\rangle_\mathcal{E}-\frac{1}{\varepsilon}\langle f,\eta\rangle_\V \; dt\leq 0.
	\end{align}Then it follows that $ \forall i\in V\: \forall t\in[0,T]$, \[ w_i(t) \leq u_i(t). \]
\end{thm}
\begin{proof} We follow the proof in \cite{CE}.
Let $z:= w-u$ and $T^*\in[0,T]$. Taking $\eta =z_+ + u\in\V_{[0,1],t\in(0,T)}$\footnote{Since either $\eta_i(t)=u_i(t)\in[0,1]$ or $\eta_i(t) = w_i(t)$ in the case when $0\leq u_i(t)\leq w_i(t)\leq 1$. } and then integrating \eqref{ACobs3} gives \[
\int_0^{T^*}\left\langle\frac{du}{dt}-\frac{1}{\varepsilon}u+\frac{1}{2\varepsilon}\mathbf{1},z_+\right\rangle_\V+\left\langle\nabla u,\nabla z_+\right\rangle_\mathcal{E}\; dt\ge0
\] 
and taking $\eta(t) = z_+(t)$ for $t<T^*$ and $\eta(t) = 0$ thereafter in \eqref{compint} gives
 \[ \int_0^{T^*}  \left\langle\frac{dw}{dt}+\frac{1}{2\varepsilon}\mathbf{1},z_+\right\rangle_\V + \langle \nabla w,\nabla z_+\rangle_\mathcal{E}-\frac{1}{\varepsilon}\langle f,z_+\rangle \; dt\leq 0 .\] Thus combining these inequalities we get
 \begin{equation*}\int_0^{T^*}  \left\langle\frac{dz}{dt},z_+\right\rangle_\V +  \langle \nabla z,\nabla z_+\rangle_\mathcal{E}\; dt \leq \frac{1}{\varepsilon}\int_0^{T^*}\ip{f-u,z_+}\; dt. \tag{$*$}\end{equation*}
 Now by subtracting $\ip{u,z_{+}}$ from \eqref{fcon}, \[ RHS \:(*)\leq \frac{1}{\varepsilon}\int_0^{T^*}\ip{z,z_+}\; dt =\frac{1}{\varepsilon}\int_0^{T^*}||z_+||_\V^2\; dt\] 
and by Proposition \ref{usefullem} we have that \[ LHS \: (*)\geq \frac{1}{2}||z_+(T^*)||_\V^2 -\frac{1}{2}||z_+(0)||_\V^2 + \int_0^{T^*}  \langle \nabla z_+,\nabla z_+\rangle_\mathcal{E}\; dt. \] 
Thus let $Z(t) := ||z_+(t)||_\V^2$. Note that $Z(0) = 0$ and thus we have by the above that \[\frac{1}{2}Z(T^*)\leq \frac{1}{2}Z(T^*) + \int_0^{T^*}  ||\nabla z_+||^2_\mathcal{E}\; dt\leq LHS \: (*)\leq RHS \: (*)\leq \frac{1}{\varepsilon}\int_0^{T^*}Z(t)\; dt.\]
Then by Gr\"onwall's integral inequality $Z(T^*)\leq 0$ for all $T^*\in [0,T]$ and hence $Z(t) = 0$ in $[0,T]$, thus for $t\in [0,T]$ $z_+(t)=0$ and so $w(t)\leq u(t)$. 
\end{proof}

\end{document}